\documentclass[oneside,reqno]{amsart}
\usepackage{etex}
\usepackage[T1]{fontenc}

\usepackage[pagebackref]{hyperref}

\renewcommand*{\backref}[1]{}
\renewcommand*{\backrefalt}[4]{%
  \ifcase #1 %
No citations.
  \or
(cit. on p. #2).%
  \else
(cit on pp. #2).%
  \fi%
}

\usepackage{amssymb,amsfonts,amsmath}
\usepackage{comment}
\usepackage[all,arc]{xy}
\usepackage{enumerate}
\usepackage{mathrsfs,mathtools}
\usepackage{todonotes,booktabs}
\usepackage{stmaryrd}
\usepackage{marvosym}
\usepackage{graphicx}
\usepackage{pdflscape}

\usepackage{bbm}
\usepackage{tikz}
\usepackage{tikz-cd}
\usetikzlibrary{trees}
\usetikzlibrary[shapes]
\usetikzlibrary[arrows]
\usetikzlibrary{patterns}
\usetikzlibrary{fadings}
\usetikzlibrary{backgrounds}
\usetikzlibrary{decorations.pathreplacing}
\usetikzlibrary{decorations.pathmorphing}
\usetikzlibrary{positioning}

\usepackage{xcolor} 
\usepackage{graphicx}

\definecolor{dark-red}{rgb}{0.5,0.15,0.15}
\definecolor{dark-blue}{rgb}{0.15,0.15,0.6}
\definecolor{dark-green}{rgb}{0.15,0.6,0.15}
\hypersetup{
    colorlinks, linkcolor=dark-red,
    citecolor=dark-blue, urlcolor=dark-green
}
\newcommand{\iHom}{\underline{\operatorname{Hom}}}

\newcommand{\cX}{\cal{X}}
\newcommand{\cY}{\cal{Y}}

\usepackage{subfiles}

\usepackage[nameinlink,capitalise,noabbrev]{cleveref}

\newtheorem{thm}{Theorem}[section]
\newtheorem{cor}[thm]{Corollary}
\newtheorem{prop}[thm]{Proposition}
\newtheorem{lem}[thm]{Lemma}

\newtheorem{quest}[thm]{Question}

\newtheorem{thmx}{Theorem}

\theoremstyle{definition}
\newtheorem{defn}[thm]{Definition}

\newtheorem{ex}[thm]{Example}

\theoremstyle{remark}
\newtheorem{rem}[thm]{Remark}

\newtheorem*{thm*}{Theorem}

\makeatletter
\let\c@equation\c@thm
\makeatother
\numberwithin{equation}{section}

\DeclareMathOperator{\Infl}{Infl}
\DeclareMathOperator{\Sp}{Sp}
\DeclareMathOperator{\Hom}{Hom}

\DeclareMathOperator{\End}{End}

\DeclareMathOperator{\cA}{\mathcal{A}}
\DeclareMathOperator{\cC}{\mathcal{C}}
\DeclareMathOperator{\cD}{\mathcal{D}}

\DeclareMathOperator{\rank}{rank}

\DeclareMathOperator{\cF}{\mathcal{F}}

\DeclareMathOperator{\Ext}{Ext}

\DeclareMathOperator{\Spec}{Spec}
\DeclareMathOperator{\Mod}{Mod}

\DeclareMathOperator{\Sub}{Sub}

\DeclareMathOperator{\Kos}{Kos}

\DeclareMathOperator{\GrAb}{GrAb}
\DeclareMathOperator{\Loc}{Loc}

\DeclareMathOperator{\Ind}{Ind}

\DeclareMathOperator{\res}{res}
\DeclareMathOperator{\Fun}{Fun}

\newcommand{\Q}{\mathbb{Q}}
\DeclareMathOperator{\Res}{Res}

\DeclareMathOperator{\Alg}{Alg}
\DeclareMathOperator{\CAlg}{CAlg}

\DeclareMathOperator{\unit}{\mathbbm{1}}

\DeclareMathOperator{\free}{free}

\usepackage{enumitem}

\newcommand{\cal}{\mathcal}
\newcommand{\xr}{\xrightarrow}

\newcommand{\Z}{\mathbb{Z}}

\Crefname{figure}{Figure}{Figures}
\Crefname{assu}{Assumption}{Assumptions}
\Crefname{prop}{Proposition}{Propositions}
\Crefname{lem}{Lemma}{Lemmas}
\Crefname{thm}{Theorem}{Theorems}
\Crefname{ex}{Example}{Examples}
\Crefname{cor}{Corollary}{Corollaries}

\newcommand{\recollement}[5]{
\xymatrix{{#1} \ar[r]|-{#2} & #3 \ar[r]|-{#4} \ar@<1ex>[l]^-{{#2}_!} \ar@<-1ex>[l]_-{{#2}^*} & #5, \ar@<1ex>[l]^-{{#4}!} \ar@<-1ex>[l]_-{{#4}^*}
}}
\let\lim\relax

\DeclareMathOperator{\lim}{lim}

\newcommand{\cU}{\mathcal{U}}

\DeclareMathOperator{\Map}{Map}
\title{Rational local systems and connected finite loop spaces}
\author{Drew Heard}
\address{Department of Mathematical Sciences, Norwegian University of Science and Technology, Trondheim}
\email{drew.k.heard@ntnu.no}

\DeclareMathOperator{\Sym}{Sym}

\date{\today}

\begin{document}

\begin{abstract}
Greenlees has conjectured that the rational stable equivariant homotopy category of a compact Lie group always has an algebraic model. Based on this idea, we show that the category of rational local systems on a connected finite loop space always has a simple algebraic model. When the loop space arises from a connected compact Lie group, this recovers a special case of a result of Pol and Williamson about rational cofree $G$-spectra. More generally, we show that if $K$ is a closed subgroup of a compact Lie group $G$ such that the Weyl group $W_GK$ is connected, then a certain category of rational $G$-spectra `at $K$' has an algebraic model. For example, when $K$ is the trivial group, this is just the category of rational cofree $G$-spectra, and this recovers the aforementioned result. Throughout, we pay careful attention to the role of torsion and complete categories.
\end{abstract}
\setcounter{tocdepth}{1}
\maketitle
\tableofcontents

\section{Introduction}
The category of non-equivariant rational spectra is very simple; it is equivalent to the derived category of $\mathbb{Q}$-modules. Greenlees has conjectured that for a compact Lie group $G$, the category of rational equivariant $G$-spectra is equivalent to the derived category of an abelian category $\cal{A}(G)$ \cite[Conjecture 6.1]{greenlees_conjecture}.  For example, when $G$ is a finite group, the conjecture holds, and is relatively elementary to prove \cite[Appendix A]{GreenleesMay1995Generalized}. The conjecture has also been proved in various other cases including (but not limited to) tori \cite{GreenleesShipley2018algebraic}, $O(2)$ \cite{Barnes2017Rational}, and $SO(3)$ \cite{Kolhkedziorek2017algebraic}. In these cases, we say that the category of rational $G$-equivariant spectra has an algebraic model. One can additionally ask for more structure to be preserved, for example one can ask for an equivalence of symmetric monoidal categories.

Inside the category of $G$-spectra sit the category of free and cofree (or Borel complete) $G$-spectra. The category of free $G$-spectra consists of those $G$-spectra that can be constructed from free cells $\Sigma^{\infty}_+G$. More specifically, it can be constructed as the localizing subcategory inside $G$-spectra generated by $\Sigma^{\infty}_+G$. Equivalently, these are the $G$-spectra for which $EG_+ \otimes X \to X$ is an equivalence, where $EG_+$ is the suspension spectra of the universal free $G$-space (see \Cref{sec:torsion_spectra}). The category of cofree $G$-spectra is the Bousfield localization of $\Sp_G$ at $\Sigma^{\infty}_+G$, or equivalently the $G$-spectra for which $X \to F(EG_+,X)$ is an equivalence. Similarly, we can construct the categories of free and cofree rational $G$-spectra, which we denote by $\Sp_{G,\Q}^{\mathrm{free}}$ and $\Sp_{G,\Q}^{\mathrm{cofree}}$, respectively. In fact, these categories are equivalent, although not by the identity functor. These categories fit into a general construction of torsion and complete categories, see \Cref{sec:torsion_completion}.

It is reasonable to conjecture that there is an algebraic model for these categories, and this is indeed the case \cite{GreenleesShipley2011algebraic,GreenleesShipley2014algebraic,pol_williamson}. We state the result for a connected compact Lie group, however we note that the cited results consider more generally arbitrary compact Lie groups.
\begin{thm}(Greenlees--Shipley, Pol--Williamson)\label{thm:rational_models}
  Let $G$ be a connected compact Lie group and $I$ be the augmentation ideal of $H^*(BG)$. Then there are Quillen equivalences
  \[
\Sp_{G,\Q}^{\mathrm{free}}\simeq  \Mod_{H^*(BG),\mathrm{inj}}^{I-\mathrm{tors}} \quad \text{ and } \quad \Sp_{G,\Q}^{\mathrm{cofree}}\simeq_{\otimes}  \Mod_{H^*(BG),\mathrm{proj}}^{I-\mathrm{comp}}
  \]
 \end{thm}
 Here the categories $\Mod_{H^*(BG),\mathrm{inj}}^{I-\mathrm{tors}}$ and $\Mod_{H^*(BG),\mathrm{proj}}^{I-\mathrm{comp}}$ are the categories of $I$-torsion dg-$H^*(BG)$-modules and $L_0^I$-complete dg-$H^*(BG)$-modules respectively, equipped with an injective and projective module category structure, respectively (see \Cref{sec:algebraic_torsion}). Moreover, the second equivalence is even shown to be symmetric monoidal.\footnote{Throughout, we indicate such an equivalence by the symbol $\simeq_{\otimes}$.}

In fact, Greenlees and Shipley have given two proofs for the equivalence between free $G$-spectra and torsion $H^*(BG)$-modules when $G$ is a \emph{connected} compact Lie group. The first \cite{GreenleesShipley2011algebraic} passes from equivariant homotopy to algebra almost immediately, while the second \cite{GreenleesShipley2014algebraic} (which also deals with the non-connected case) stays in the equivariant world as long as possible. As noted by the authors, staying in the equivariant worlds seems to help the extension to the non-connected case. In the cofree case, the authors also stay in the equivariant world as long as possible. Our approach is to move away from equivariant homotopy immediately, and as such is closer in spirit to the original proof of Greenlees and Shipley. Indeed, we begin with the observation that there is a symmetric monoidal equivalence of $\infty$-categories
\begin{equation}\label{eq:borel}
\Sp_{G,\Q}^{\mathrm{cofree}} \simeq_{\otimes} \Fun(BG,\Mod_{H\Q}),
\end{equation}
see \Cref{prop:cofree_global}, where $\Fun(-,-)$ denotes the $\infty$-category of functors and $BG$ is considered as an $\infty$-groupoid. We call this the $\infty$-category of rational local systems on $BG$.

An advantage of moving away from equivariant homotopy is that one can work more generally. For a space $Y$ (again thought of as an $\infty$-groupoid) we let $\Loc_{H\Q}(Y) = \Fun(Y,\Mod_{H\Q})$ denote the $\infty$-category of rational local systems on $Y$.
\begin{quest}
  For which spaces $Y$ does the $\infty$-category $\Loc_{H\Q}(Y)$ have an algebraic model?
\end{quest}
The above results show that this is true whenever $Y = BG$ for a compact Lie group $G$. A connected compact Lie group is a particular example of a connected finite loop space. Our first main result is the following.
\begin{thmx}[\Cref{thm:main}]  Let $X$ be a connected finite loop space, then there is an equivalence of symmetric monoidal $\infty$-categories
  \[
\Loc_{H\Q}(BX) \simeq_{\otimes } \cal{D}(\Mod^{I-\mathrm{comp}}_{H^*(BX)}).
  \]
\end{thmx}
Here $\cal{D}(\Mod^{I-\mathrm{comp}}_{H^*(BX)})$ is the symmetric monoidal stable $\infty$-category underlying the category $\Mod^{I-\mathrm{comp}}_{H^*(BX),\mathrm{proj}}$ of $L_0^I$-complete dg-$H^*(BX)$-modules, again equipped with the projective model structure. We note that there do indeed exist connected finite loop spaces not rationally equivalent to compact Lie groups \cite{AndersenBauerGrodalPedersen2004finite}. The key fact is that the rational cohomology of the classifying space of any connected finite loop space is a polynomial algebra
\[
H^*(BX;\Q) \cong \Q[y_1,\ldots,y_r]
\]   
where the generator $y_i$ is in (even) degree $2d_i$. In fact, the integers $\{ d_1,\ldots,d_r \}$ uniquely determine the rational homotopy type of the finite loop space $X$. This is the key computational result that we need, along with the fact that rational polynomial rings are formal.  Finally, we note that such an $X$ is always homotopy equivalent to a manifold by the main result of \cite{BauerKitchlooNotbohmPedersen2004Finite}. 

The proof proceeds through a series of equivalences of symmetric monoidal stable $\infty$-categories, as indicated below.
\[
\Loc_{H\Q}(BX) \stackrel{(\ref{prop:unipotence_loops})}{\simeq_{\otimes } } L_{H\Q}\Mod_{C^*(BX;\Q)}\stackrel{(\ref{prop:completions})}{\simeq_{\otimes } } \Mod_{C^*(BX;\Q)}^{I-\mathrm{comp}} \stackrel{(\ref{lem:completion_compare})}{\simeq_{\otimes } } \cD_{H^*(BX)}^{I-\mathrm{comp}}  \stackrel{(\ref{thm:completion_monoidal_equiv})}{\simeq_{\otimes } } \cD(\Mod_{H^*(BX)}^{I-\mathrm{comp}}).
\]
The first equivalence relies on the concept of a unipotent stable $\infty$-category introduced in \cite{MathewNaumannNoel2017Nilpotence}, and relies heavily on their work. As we explain in \Cref{rem:unipotence_cg}, one could also deduce the result using the compactly generated localization principle of Pol and Williamson \cite[Theorem 3.14]{pol_williamson}, or \Cref{prop:compact_gen_loc} in this paper.

In equivariant homotopy we work with a bit more generality than with just free and cofree $G$-spectra. For a closed subgroup $K$ of $G$, we define $\infty$-categories $\Sp_{G,\langle K \rangle}$ and $\Sp_G^{\langle K \rangle}$ of $G$-spectra `at $K$', as well as their rationalized versions. The terminology is used because a non-trivial $G$-spectrum $M \in \Sp_{G,\langle K \rangle}$ if and only if its geometric isotropy is exactly $K$, i.e., its geometric $K$-fixed points are non-trivial, and its geometric $H$-fixed points are trivial for all $H \ne K$. The rational categories $\Sp_{G,\langle K \rangle,\Q}$ also appear in the computation of the localizing tensor-ideals of $\Sp_{G,\Q}$ \cite{Greenlees2019Balmer}; these are precisely the minimal localizing tensor-ideals. Finally, we note that the categories $\Sp_{G}^{\langle K \rangle}$ appear naturally in the work of Ayala--Mazel-Gee--Rozenblyum \cite{1910.14602} and Balchin--Greenlees \cite{1903.02669}, see \Cref{rem:categories}.

Our second theorem is the following.
\begin{thmx}[\Cref{cor:pol_williamson,cor:free_local}]
  Let $G$ be a compact Lie group, and $K$ a closed subgroup such that the Weyl group $W_GK = N_GK/K$ is a connected compact Lie group, then there are equivalences of stable $\infty$-categories
  \[
\Sp_{G,\langle K \rangle, \Q} \simeq \cD(\Mod_{H^*(B(W_GK))}^{I-\mathrm{tors}}) \quad \text{ and } \quad \Sp_{G,\Q}^{\langle K \rangle} \simeq_{\otimes} \cD(\Mod_{H^*(B(W_GK))}^{I-\mathrm{comp}}).
  \]
\end{thmx}
When $G$ is a connected compact Lie group and $K$ is the trivial subgroup this recovers \Cref{thm:rational_models}. When $G$ is an arbitrary compact Lie group and $K = G$, then $\Sp_{G,\Q,\langle G \rangle} \simeq \Sp_{G,\Q}^{\langle G \rangle} \simeq_{\otimes} \Sp_{\Q}$, the ordinary category of rational non-equivariant spectra, and this is just the statement that the rational stable homotopy category is equivalent to the derived category of $\Q$-vector spaces.

We finish by constructing an Adams spectral sequence in the category $\Loc_{H\Q}(BX)$ for $X$ a connected finite loop space. In fact, we show that the Adams spectral sequence can easily be constructed using the universal coefficient spectral sequence for ring spectra \cite[Theorem IV.4.1]{ElmendorfKrizMellMay1997Rings}.

\subsection*{Acknowledgements}
We thank Tobias Barthel, Markus Land and Denis Nardin for helpful conversations, the referee for their helpful comments and suggestions, and the SFB Higher Invariants 1085 in Regensburg for support. We were also supported in part by NTNU Trondheim and grant number TMS2020TMT02 from the Trond Mohn Foundation.
\subsection*{Conventions}
We work throughout mainly with $\infty$-categories although some results need to be translated from model categories to $\infty$-categories; in \Cref{sec:appendix} we give a very brief recap of what we need, as well as references to more detailed accounts.

An adjunction $F \colon \cC \leftrightarrows \cD \colon G$ between symmetric monoidal stable $\infty$-categories will be called symmetric monoidal if $F$ is a symmetric monoidal functor. Note that in this case $G$ automatically acquires the structure of a lax symmetric monoidal functor \cite[Corollary 7.3.2.7]{ha}.

For a compact Lie group $G$, we will write $\Sp_G$ for the $\infty$-category of $G$-equivariant spectra; in the non-equivariant case, we write $\Sp$. For a space $X$, and an $\infty$-category $\cC$, we will write $\Fun(X,\cC)$ for the $\infty$-category of functors from $X$ to $\cC$, where $X$ is thought of as an $\infty$-groupoid. For example, when $X = BG$, the category $\Fun(BG,\cC)$ denotes the $\infty$-category of objects in $\cC$ with a $G$-action.

A localizing category $\cD$ of $\cC$ is a full, stable, subcategory of $\cC$ that is closed under extension, retracts, and filtered colimits. It is additionally an ideal if $X \in \cD$ and $Y \in \cC$ implies $X \otimes Y \in \cD$. Given a collection of objects $\{ X_i \}_{i \in I} \in \cC$ we will write $\Loc(\{ X_i \mid i \in I\})$ for the smallest localizing subcategory of $\cC$ containing each $X_i$. In the case of a single object $X$, we simply write $\Loc(X)$. 

Finally, if $\cC$ is a closed symmetric monoidal category with internal hom object $F(-,-)$ and monoidal unit $\unit$, then we write $\mathbb{D}X = F(X,\unit)$ for the internal dual of an object $X$. 
\section{Completion and torsion in algebra and topology}
We begin by reviewing the construction of torsion and complete categories in a symmetric monoidal stable $\infty$-category. We consider torsion and completion for ring spectra and dg-algebras, and relate the latter to algebraic categories of torsion and complete objects.
\subsection{Torsion and complete objects}\label{sec:torsion_completion}
We recall the basics of  torsion and complete objects in a symmetric monoidal presentable stable $\infty$-category $(\cC,\otimes,\unit)$. For simplicity, we assume that $\cC$ is compactly generated by dualizable objects. Note that our assumptions imply that $\cC$ is closed monoidal, and we write $\iHom_{\cC}(-,-)$ for the internal Hom object in $\cC$. They also imply that all compact objects are dualizable \cite[Lemma 2.5]{bhv1} (with the converse holding if the unit $\unit$ is compact). The theory in this section goes back to (at least) Hovey--Palmieri--Strickland \cite{HoveyPalmieriStrickl1997Axiomatic}, and has also been considered by Dwyer--Greenlees \cite{DwyerGreenlees2002Complete}, Mathew--Naumann--Noel \cite{MathewNaumannNoel2017Nilpotence}, and Barthel--Heard--Valenzuela \cite{BarthelHeardValenzuela2018Local}.

We consider three full subcategories of $\cC$ defined in the following way.
\begin{defn}
  Let $\cal{A} = \{A_i\} $ be a set of compact (and hence dualizable) objects of $\cC$.
\begin{enumerate}
  \item We say that $M \in \cC$ is $\cA$-torsion if it is in the localizing subcategory of $\cC$ generated by the set $\cA$. We let $\cC^{\cA-\mathrm{tors}} \subseteq \cC$ denote the full subcategory of $\cA$-torsion objects.
  \item We say that $M \in \cC$ is $\cA$-local if for any $N$ which is $\cA$-torsion, the space of maps $\Hom_{\cC}(N,M) \simeq 0$, or equivalently, if $A_i \otimes M \simeq 0$ for each $A_i \in \cal{A}$ \cite[Proposition 3.11]{MathewNaumannNoel2017Nilpotence}. We let $\cC^{\cA-\mathrm{loc}} \subseteq \cC$ denote the full subcategory of $\cA$-local objects.
  \item We say that $M \in \cC$ is $\cA$-complete if for any $N \in \cC$ which is $\cA$-local the space of maps $\Hom_{\cC}(N,M) \simeq 0$. We let $\cC^{\cA-\mathrm{comp}} \subseteq \cC$ denote the full subcategory of $A$-complete objects.
\end{enumerate}
   \end{defn}
\begin{rem}
  Note that we do not assume that $\cC^{\cA-\mathrm{tors}}$ is a localizing ideal, i.e., is not automatically closed under tensor products. However, in practice, we will often be in the situation where every localizing subcategory is automatically a tensor ideal (for example, this holds whenever the category $\cC$ has a single compact generator \cite[Lemma 1.4.6]{HoveyPalmieriStrickl1997Axiomatic})
\end{rem}

   The following is shown in \cite[Theorem 3.3.5]{HoveyPalmieriStrickl1997Axiomatic} or \cite[Theorem 2.21]{BarthelHeardValenzuela2018Local}.
\begin{thm}[Abstract local duality]\label{thm:hps}
Let $\cC$ and $\cA$ be as above.
\begin{enumerate}
  \item The inclusion functor $\iota_{\mathrm{tors}} \colon \cC^{\cA-\mathrm{tors}} \hookrightarrow \cC$ has a right adjoint $\Gamma_{\cA}$, and the inclusion functors $\iota_{\mathrm{loc}} \colon \cC^{\cA-\mathrm{loc}} \hookrightarrow \cC$ and $\iota_{\mathrm{comp}} \colon \cC^{\cA-\mathrm{comp}} \to \cC$ have left adjoints $-[\cA^{-1}]$ and $\Lambda^{\cA}$, respectively.
  \item There are cofiber sequences
  \[
\Gamma_{\cA}X \to X \to X[\cA^{-1}]
  \]
  and
  \[
\Delta_{\cA}(X) \to X \to \Lambda^{\cA}X
  \]
  for all $X \in \cC$. In particular, $\Gamma_{\cA}$ is a colocalization functor and both $-[\cA^{-1}]$ and $\Lambda^{\cA}$ are localization functors.
  \item The functors $\Lambda^{\cA} \colon \cC^{\cA-\mathrm{tors}} \to \cC^{\cA-\mathrm{comp}}$ and $\Gamma_{\cA} \colon \cC^{\mathrm{comp}} \to \cC^{\mathrm{tors}}$ are mutually inverse equivalences of stable $\infty$-categories.
  \item Considered as endofunctors of $\cC$, there are adjunctions
  \[
\Hom_{\cC}(\Gamma_{\cA}X,Y) \simeq \Hom_{\cC}(X,\Lambda^{\cA}Y)
  \]
  and
  \[
\iHom_{\cC}(\Gamma_{\cA}X,Y) \simeq \iHom_{\cC}(X,\Lambda^{\cA}Y)
  \]
  between $\Gamma_A$ and $\Lambda^A$.
\end{enumerate}
\begin{rem}
  We note that the functors and categories above do not depend on the set $\cA$, but only on the thick subcategory it generates. 
\end{rem}
\begin{rem}\label{rem:smashing}
  If $\cC^{\cA-\mathrm{tors}}$ is a localizing ideal, then $\Gamma_{\cA}$ and $-[\cA^{-1}]$ are both smashing, i.e., $\Gamma_{\cA}(X) \simeq \Gamma_{\cA}(\unit) \otimes X$ and similar for $-[{\cA}^{-1}]$.
\end{rem}
\begin{rem}
  In the literature $\cal{A}$-torsion objects are also sometimes referred to as $\cal{A}$-cellular objects, for example in \cite{GreenleesShipley2013cellularization} (see in particular \cite[Proposition 2.5 and Corollary 2.6]{GreenleesShipley2013cellularization})
\end{rem}
\end{thm}
Pictorially, we can represent the functors and categories in the following digram.
\begin{equation}\label{eq:abstractdiagram}
\begin{gathered}
\xymatrix{& \cC^{\cA-\mathrm{loc}} \ar@<0.5ex>[d] \ar@{-->}@/^1.5pc/[ddr] \\
& \cC \ar@<0.5ex>[u]^{-[\cA^{-1}]} \ar@<0.5ex>[ld]^{\Gamma_{\cA}} \ar@<0.5ex>[rd]^{\Lambda^{\cA}} \\
\cC^{\cA-\mathrm{tors}} \ar@<0.5ex>[ru] \ar[rr]_{\sim} \ar@{-->}@/^1.5pc/[ruu] & & \cC^{\cA-\mathrm{comp}}, \ar@<0.5ex>[lu]}
\end{gathered}
\end{equation}
Each of the pairs $(\cC^{\cal{A}-\mathrm{tors}},\cC^{\cal{A}-\mathrm{loc}})$ and $(\cC^{\cal{A}-\mathrm{loc}},\cC^{\cal{A}-\mathrm{comp}})$ form a \emph{semi-orthogonal} decomposition of $\cC$ in the sense of \cite[Definition 7.2.0.1]{sag}.

We note the following, which is \cite[Proposition 2.34]{BarthelHeardValenzuela2018Local}.
\begin{lem}\label{lem:bosfield_localization}
Suppose that $A \in \cC$, and let $\cA = \{ A \otimes D \}$ where $D \in \cC$ runs over a set of compact generators of $\cC$.\footnote{This conditions forces $\cC^{\cA-\mathrm{loc}}$ to be the localizing tensor ideal generated by $A$ \cite[Lemma 1.4.6]{HoveyPalmieriStrickl1997Axiomatic}.} The inclusion $\cC^{\cA-\mathrm{comp}} \hookrightarrow \cC$ has a left adjoint given by Bousfield localization at $\cA$, i.e., $\cC^{\cA-\mathrm{comp}} \simeq_{\otimes} L^{\cC}_{\cA}$.
\end{lem}
We now present a simplified version of the Greenlees--Shipley cellularization principle \cite[Corollary 2.7]{GreenleesShipley2013cellularization} that suffices for our purposes.

\begin{prop}[Greenlees--Shipley]\label{prop:cellularization_principle}
  Let $\cC$ and $\cD$ be stable $\infty$-categories, and let
  \[
\begin{tikzcd}
F \colon \cC \arrow[r, shift left] & \cD \colon G \arrow[l, shift left]
\end{tikzcd}
  \]
be an adjunction.
\begin{enumerate}
  \item Let $K$ be in $\cC$ and suppose that the following hold:
  \begin{enumerate}
    \item $K$ is compact is $\cC$, and $F(K)$ is compact in $\cD$.
    \item The unit $\eta \colon K \to GF(K)$ is a natural isomorphism.
  \end{enumerate}
  Then, there is an equivalence of $\infty$-categories
  \[
\cC^{K-\mathrm{tors}} \simeq \cD^{F(K)-\mathrm{tors}}.
  \]
  \item Let $L$ be in $\cD$ and suppose that the following hold:
  \begin{enumerate}
    \item $L$ is compact in $\cD$, and $G(L)$ is compact in $\cC$.
    \item The counit $\epsilon \colon FG(L) \to L$ is a natural isomorphism.
  \end{enumerate}
  Then, there is an equivalence of $\infty$-categories
  \[
\cC^{G(L)-\mathrm{tors}} \simeq \cD^{L-\mathrm{tors}}.
  \]
\end{enumerate}
\end{prop}
\begin{proof}
We prove (1), and leave the minor adjustments for (2) to the reader.  We first claim that $(F,G)$ gives rise to an adjunction
\begin{equation}\label{eq:ss_cell}
\begin{tikzcd}
F' \colon \cC^{K-\mathrm{tors}}  \arrow[r, shift left] & \cD^{K-\mathrm{tors}}  \colon G' \arrow[l, shift left]
\end{tikzcd}
\end{equation}
Indeed, because $F$ preserves colimits, $F(\Loc(L)) \subseteq \Loc(F(K))$, see, for example, \cite[Lemma 2.5]{bchv1}. We can therefore take $F'$ to be the restriction of $F$ to $\Loc(K)$. Setting $G' = \Gamma_{K}G$, one verifies that $(F',G')$ form an adjoint pair, which we claim is an equivalence.

Indeed, consider the full subcategory of $\cC^{\mathrm{tors}}$ consisting of those $X$ for which the unit $X \to GF(X)$ is an equivalence. This is a localizing subcategory containing $K$ by assumption. Since $K$ generates $\cC^{K-\mathrm{tors}}$ this localizing subcategory is all of $\cC^{\mathrm{tors}}$. Likewise, the full subcategory of $\cD^{\mathrm{tors}}$ consisting of those $Y$ for which the counit $FG(Y) \to Y$ is an equivalence, is localizing. Moreover, it contains $F(K)$ by the triangle identities, and hence is equal to $\cD^{F(K)-\mathrm{tors}}$. 
\end{proof}
A sort of dual result, due to Pol and Williamson, is the compactly generated localization principle \cite[Theorem 3.14]{pol_williamson}. Again, we only prove a special case of their theorem which will suffice for our purposes.
\begin{prop}[Pol--Williamson]\label{prop:compact_gen_loc}
  Let $\cC$ and $\cD$  be symmetric monoidal stable $\infty$-categories and
  \[
\begin{tikzcd}
F \colon \cC \arrow[r, shift left] & \cD \colon G \arrow[l, shift left]
\end{tikzcd}
  \]
  a symmetric monoidal adjunction.
  \begin{enumerate}
    \item Let $E \in \cC$ and suppose that the following hold:
  \begin{enumerate}
    \item $L_E\cC$ is compactly generated by $K$ and $L_{F(E)}\cal{D}$ is compactly generated by $F(K)$.
    \item The unit map $\eta_{K} \colon K \to GF(K)$ is an equivalence.
  \end{enumerate}
  Then, there is a symmetric monoidal equivalence of $\infty$-categories
  \[
L_E\cC \simeq_{\otimes} L_{F(E)}\cD
  \]
  \item Let $E' \in \cD$ and suppose that the following hold:
  \begin{enumerate}
    \item $L_{E'}\cD$ is compactly generated by $L$ and $L_{G(E')}\cal{D}$ is compactly generated by $G(L)$.
    \item The counit maps $\epsilon_{L} \colon FG({L}) \to {L}$ and $\epsilon_{E'} \colon FG(E') \to E'$ are equivalences.
  \end{enumerate}
  Then, there is a symmetric monoidal equivalence of $\infty$-categories
  \[
L_{G(E')}\cC \simeq_{\otimes} L_{E'}\cD
  \]
  \end{enumerate}
\end{prop}
\begin{proof}
We prove (1); the proof for (2) is similar - the extra assumption is only used to ensure that the adjunction descends to the localized categories, as we now describe in (1).

First observe that if $Y \in \cC$ is $E$-acyclic, then $F(Y) \in \cD$ is $F(E)$-acyclic because $F$ is a symmetric monoidal functor. We claim it follows that if $N \in L_{F(E)}\cD$, then $G(N) \in L_E\cC$. To see this, choose an $E$-acyclic $Y$, then we must show that $\Hom_{\cC}(Y,G(N)) \simeq \ast$. But $\Hom_{\cC}(Y,G(N)) \simeq \Hom_{\cD}(F(Y),N) \simeq \ast$ because $F(Y)$ is $F(E)$-acyclic and $N \in L_{F(E)}\cD$ by assumption.

  Let $ F' = L_{F(E)} \circ F$, then by inspection we have a symmetric monoidal adjunction $ F' \colon L_A \cC \leftrightarrows L_{F(E)}\cD \colon G'$, where $G'$ is the restriction of $G$ to $L_{F(E)}\cD$, which we claim is an equivalence.

First, because $F( K) \in L_{F(E)}\cD$ it is not hard to see that assumption (b) implies that the unit map $ \eta'_{ K} \colon  K \to G' F'( K)$ is also an equivalence. Note that $ F'$ preserves colimits, and since it preserves compact objects by assumption (a), its right adjoint $G'$ preserves colimits as well. It follows that the unit is always an equivalence, and that $F'$ is fully-faithful.

It then follows from the triangle identities that the counit $F'G'(F({K})) \to F( K)$ is also an equivalence, and a localizing subcategory argument shows then that the counit is always an equivalence. Hence, $G'$ is also fully faithful, and $(F',G')$ is an adjoint equivalence as claimed.
\end{proof}
\subsection{Torsion and completion for graded commutative rings}
Throughout this section we fix a graded commutative ring $A$, and let $\Mod_A$ denote the category of dg-$A$-modules. We can give this category the projective model structure \cite[Theorem 3.3]{BarthelMayRiehl2014Six} with weak equivalences the quasi-isomorphisms, fibrations degreewise surjections, and cofibrations the subcategory of maps which have the left lifting property with respect to every map which is simultaneously a fibration and a weak equivalence. This is a compactly generated (in the sense of \cite[Definition 6.5]{BarthelMayRiehl2014Six}) monoidal model category, and we write $\cal{D}_A$ for the associated symmetric monoidal stable $\infty$-category (see \Cref{sec:appendix} for a very brief summary of the translation between model categories and $\infty$-categories).

  We can also give $\Mod_A$ the injective model structure with weak equivalences the quasi-isomorphisms, cofibrations degreewise monomorphisms, and fibrations those maps which have the right lifting property with respect to every map that is simultaneously a cofibration and a weak equivalence. Because the weak equivalences are the same as in the projective model structure, the underlying $\infty$-category $\cal{D}_A$ does not depend on which model structure we use. However, the injective model structure is not monoidal, and so from this perspective one does not see the symmetric monoidal structure on $\cal{D}_A$.

  For any $x \in A$, we define the unstable Koszul complex as
  \[
K(x) = \operatorname{fib}(\Sigma^{|x|}A \xr{\cdot x} A),
  \]
  where the fiber is taken in $\cal{D}_A$, and the stable Koszul complex
  \[
K_{\infty}(x) = \operatorname{fib}(A \to A[x^{-1}])
  \]
  where, as usual, $A[x^{-1}]$ is defined as the colimit of the multiplication by $x$ map.

  Let $I = (x_1,\ldots,x_n)$ be a finitely generated ideal, and then define
  \[
K(I) = K(x_1) \otimes_A \cdots \otimes_A K(x_n) \quad \text{ and } \quad K_{\infty}(I) = K_{\infty}(x_1) \otimes_A \cdots \otimes_A K_{\infty}(x_n).
  \]
  \begin{defn}
    Let $\cD_A^{I-\mathrm{tors}}$ denote the localizing subcategory of $A$ generated by the compact object $K(I)$.
  \end{defn}
Accordingly, applying the general machinery of \Cref{sec:torsion_completion}, we have the following categories and functors:
\[
\begin{split}
\Gamma_I &\colon \cD_A \to \cD_A^{I-\mathrm{tors}}\\
-[I^{-1}]& \colon \cD_A \to \cD_A[A^{-1}]\\
\Lambda^I& \colon \cD_A \to \cD_A^{I-\mathrm{comp}},
\end{split}
\]
as well as an equivalence of $\infty$-categories $\cD_A^{I-\mathrm{tors}} \simeq \cD_A^{I-\mathrm{comp}}$. 
\begin{rem}\label{rem:torsion_functor}
  As shown in \cite[Section 6]{DwyerGreenlees2002Complete}, we have $\Gamma_A(-) \simeq K_{\infty}(I) \otimes_A -$, and hence $\Lambda^A(-) \simeq \Hom_A(K_{\infty}(I),-)$ by local duality.
\end{rem}

\begin{rem}\label{rem:ideal_torsion}
  The notation $-[I^{-1}]$ is suggestive. Indeed, suppose that $I = (x_1)$ is principal, then it is straightforward to see that $M[I^{-1}] \simeq M[x_1^{-1}] \simeq M \otimes A[x_1^{-1}]$. In fact, $\cD_A^{I-\mathrm{loc}} \simeq \cD_{A[x_1^{-1}]}$.  More generally, $M[I^{-1}] \simeq \bigotimes_{i=1}^n M[x_i^{-1}]$, where the tensor product is taken in $\cD_A$.  In particular, we see that $M \in \cD_A^{I-\mathrm{tors}}$ if and only if $M[x_i^{-1}] \simeq 0$ for $1 \le i \le n$. This characterization will prove useful later.
\end{rem}

\begin{rem}\label{rem:homology_check}
  The categories $\cD_A^{I-\mathrm{tors}}$ and $\cD_A^{I-\mathrm{comp}}$ can both be characterized purely homologically. Indeed, using the local cohomology and homology spectral sequences (see \cite[Proposition 3.20]{BarthelHeardValenzuela2018Local} or \cite[Section 6]{DwyerGreenlees2002Complete}) one sees that
\[
\begin{split}
\cD_A^{I-\mathrm{tors}} &= \{ M \in \cD_A \mid H_*M \text{ is } I-\mathrm{torsion}\}\\
\cD_A^{I-\mathrm{comp}} &= \{ M \in \cD_A \mid H_*M \text{ is } L_0^I-\mathrm{complete} \}
\end{split}
\]
where the $I$-torsion and $L_0^I$-completion are discussed in more detail in \Cref{sec:algebraic_torsion}. 
\end{rem}
\subsection{Algebraic torsion and completion for graded rings}\label{sec:algebraic_torsion}
In this section, we compare the categories constructed via local duality in the previous section with derived categories of certain abelian categories. We now suppose that $A$ is Noetherian, and that $I$ is generated by a regular sequence. These assumptions can be weakened; it would suffice to take $A$ to be a commutative ring and $I$ to be a weakly proregular sequence (see \cite[Definition 3.21]{PortaShaulYekutieli2014homology}),  however they suffice for our purposes.

  Let $I \subset A$ be an ideal, and let $\Mod_A^{I-\mathrm{tors}}$ be the abelian subcategory of $I$-torsion modules, i.e. those $M \in \Mod_A$ for which every element of the underlying graded module is annihilated by a power of $I$, see \cite{BrodmannSharp2013Local}. We note that $\Mod_A^{I-\mathrm{tors}}$ is Grothendieck abelian, see \cite[\href{https://stacks.math.columbia.edu/tag/0BJA}{Tag 0BJA}]{stacks-project} and is hence locally presentable \cite[Proposition 3.10]{MR1780498}. 

  We recall that there is an adjunction
  \[
\begin{tikzcd}
i \colon \Mod_A^{I-\mathrm{tors}} \arrow[r, shift left] & \Mod_A \colon \Gamma^0_I \arrow[l, shift left]
\end{tikzcd}
  \]
We give $\Mod_A^{I-\mathrm{tors}}$ the injective model structure induced by $\Gamma^0_I$ using \cite[Theorem 2.2.1]{HessKcedziorekRiehlShipley2017necessary} and let $\cD(\Mod_A^{I-\mathrm{tors}})$ denote the associated $\infty$-category. Note that this does not have a natural monoidal structure. The above adjunction is Quillen (where $\Mod_A$ is given the injective model structure), and so by \Cref{lem:hinich_adjoint} gives rise to an adjunction of $\infty$-categories
  \[
\begin{tikzcd}
\underline{i} \colon \cD(\Mod_A^{I-\mathrm{tors}}) \arrow[r, shift left] & \cD_A \colon \underline{\Gamma^0_I} \arrow[l, shift left]
\end{tikzcd}
  \]
  The following appears in various forms throughout the literature, e.g., \cite{DwyerGreenlees2002Complete,GreenleesShipley2011algebraic,PortaShaulYekutieli2014homology,BarthelHeardValenzuela2020Derived}.
\begin{thm}\label{thm:tors}
  There is an equivalence of $\infty$-categories
  \[
\cD(\Mod_A^{I-\mathrm{tors}}) \simeq \cD_A^{I-\mathrm{tors}}
  \]
\end{thm}

\begin{proof}
There are a number of ways to do this - we follow \cite[Section 5]{GreenleesShipley2013cellularization} and use the cellularization principle \Cref{prop:cellularization_principle}. Thus, we take $L = K(I)$ noting that this is compact in $\cD_A$. The homology of $K(I)$ is $I$-power torsion, and hence we also write $K(I)$ to refer to the same object in $\cD(\Mod_A^{I-\mathrm{tors}})$.  We observe that $K(I)$ is in fact a compact generator of $\cD(\Mod_A^{I-\mathrm{tors}})$ (see the proof of Proposition 6.1 of \cite{DwyerGreenlees2002Complete} and the discussion in the last paragraph of page 180 of \cite{GreenleesShipley2013cellularization}), so that $\Loc(K(I)) = \cD(\Mod_A^{I-\mathrm{tors}})$. Finally, the counit $\underline{i} \circ \underline{\Gamma_I^0}(K(I)) \to K(I)$ is clearly an equivalence. Thus, the cellularization principle gives an equivalence $\cD(\Mod_A^{I-\mathrm{tors}}) \simeq \cD_A^{I-\mathrm{tors}}$, as claimed.
\end{proof}
\begin{rem}
  As noted, there are other approaches to this. One other way is to show directly that $\underline{i}$ is fully faithful (see for example \cite[Theorem 1.3]{Positselski2016Dedualizing}), with essential image the full subcategory of $\cal{D}_A$ consisting of those complexes whose homology is $I$-torsion \cite[Corollary 4.32]{PortaShaulYekutieli2014homology}. By \Cref{rem:homology_check} this is precisely the category $\cD_A^{I-\mathrm{tors}}$.
  \end{rem}

We now move onto the completion functor. Here, the algebraic version of completion we use is not $I$-adic completion (which is neither left nor right exact in general) as one may expect, but rather $L_0^I$-completion, which we recall now (for a useful summary, see \cite[Appendix A]{HoveyStrickl1999Morava}). 
\begin{defn}
  Let $L_0^I$ denote the zero-th left derived functor of the (non-exact) $I$-adic completion functor, then $M$ is said to be $L_0^I$-complete if $M \to L^0_I(M)$ is an isomorphism.
\end{defn}
\begin{ex}
    In the simple case where $A = \Z$ and $I = (p)$, Bousfield and Kan defined a notion of $\Ext-p$ completeness by asking that the natural map $M \to \Ext^1_{\Z}(\Z/p^{\infty},M)$ is an isomorphism, or equivalently, that $\Hom_{\Z}(\Z[p^{-1}],M) = \Ext^1_{\Z}(\Z[p^{-1}],M) = 0$. This turns out to be equivalent to asking that $M$ is $L_0^I$ complete.
\end{ex}

For a dg-module $M$, we say that $M$ is $L_0^I$-complete if the underlying graded module is, and let $\Mod_A^{I-\mathrm{comp}}$ denote the full subcategory of $L_0^I$-complete dg-modules. There is an adjunction
  \[
\begin{tikzcd}
L_0^I \colon \Mod_A \arrow[r, shift left] & \Mod_A^{I-\mathrm{comp}} \colon i \arrow[l, shift left]
\end{tikzcd}
  \]
  which is symmetric monoidal, where the monoidal structure on $\Mod_A^{I-\mathrm{comp}} $ is given by $L_0^I(M \otimes_A N)$.

 The subcategory $\Mod_A^{I-\mathrm{comp}}$ of $L_0^I$-complete modules is abelian, but not Grothendieck, as filtered colimits are not exact. Following unpublished notes of Rezk \cite{rezk_charles}, Pol and Williamson \cite[Proposition 7.5]{pol_williamson} showed that $\Mod_A^{I-\mathrm{comp}}$ admits a projective model structure with weak equivalences the quasi-isomorphisms, fibrations degreewise surjections, and cofibrations the subcategory of maps which have the left lifting property with respect to every map which is simultaneously a fibration and a weak equivalence. This model structure is symmetric monoidal, and the above adjunction is a Quillen adjunction \cite[Proposition 7.7]{pol_williamson}, which is symmetric monoidal because $L_0^I$ is monoidal and the unit $A$ is cofibrant.

We let $\cal{D}(\Mod_A^{I-\mathrm{comp}})$ denote the underlying $\infty$-category of $\Mod_A^{I-\mathrm{comp}}$, then there is a symmetric monoidal adjunction of stable $\infty$-categories
\[
\begin{tikzcd}
\underline{L_0^I} \colon \cal{D}_A \arrow[r, shift left] & \cal{D}(\Mod_A^{I-\mathrm{comp}}) \colon \underline{i} \arrow[l, shift left]
\end{tikzcd}
\]
\begin{thm}[Pol--Williamson]\label{thm:completion_monoidal_equiv}
  There is a symmetric monoidal equivalence of $\infty$-categories
  \[
\cal{D}_{A}^{I-\mathrm{comp}} \simeq_{\otimes} \cal{D}(\Mod_A^{I-\mathrm{comp}}).
  \]
\end{thm}
\begin{proof}
As shown by Rezk \cite[Theorem 10.2]{rezk_charles}, the counit of the above adjunction is an equivalence (i.e., $\underline {i}$ is a fully-faithful functor and $\underline{L_0^I}$ is a Bousfield localization), with image these complexes whose homology is $L$-complete. The essential image is then precisely $\cD_A^{I-\mathrm{comp}}$, see \Cref{rem:homology_check}. The equivalence is symmetric monoidal because $\underline{L_0^I}$ is a symmetric monoidal functor.
\end{proof}
\subsection{An algebraic geometric description of local objects}
Let $\cX$ be a quasi-compact separated scheme, then we can associate to it the derived $\infty$-category $\cD_{qc}(\cX)$ of quasi-coherent sheaves of $\cal{O}_{\cX}$-modules \cite[Definition 1.3.5.8]{ha}. Given a morphism $f \colon \cX \to \cY$ of quasi-compact separated schemes we can define (derived) pushforward and pullback functors
\[
f_* \colon \cD_{qc}(\cX) \to \cD_{qc}(\cY) \quad \text{ and }\quad f^* \colon \cD_{qc}(\cY) \to \cD_{qc}(\cX)
\]
where the pair $(f^*,f_*)$ are adjoint.

We now continue with the notation as in the previous section, and so we fix a graded Noetherian ring $A$ and a homogeneous ideal $I = (x_1,\ldots,x_n)$. Geometrically, we let $\cal{X} = \Spec(A)$ (the spectrum of homogeneous prime ideals in the graded ring $A$), $\cal{Z} = V(I)$, the closed subset of $\cal{X}$ defined by $I$, and $\cal{U} = \cal{X} - \cal{Z}$. We then have an open immersion $j \colon \cal{U} \to \cal{X}$. We define the $\infty$-category $\cD_{qc}^{\cal{Z}}(\cal{X})$ as the full-subcategory of $\cal{D}_{qc}(\cal{X})$ consisting of those $\cal{F}$ for which $j^*\cal{F} \simeq 0$ in $\cD_{qc}(\cal{U})$.
\begin{lem}\label{lem:torsion_alg_geo}

The equivalence of categories $\cD_{qc}(\cal{X}) \simeq \cD_A$ restricts to an equivalence of $\infty$-categories
  \[
\cD^{\cal{Z}}_{qc}(\cal{X}) \simeq \cD_A^{I-\mathrm{tors}}
  \]
\end{lem}
\begin{proof}
Observe that $\cal{U}$ can be written as a union of open subschemes of the form $\Spec A[x_i^{-1}]$ for $1 \le i \le n$. Let $\cal{F}$ be in $\cD_{qc}(\cal{X})$ and let $M \in \cD_A$ denote the corresponding complex. Then $\cal{F} \in \cD_{qc}^{\cal{Z}}(\cal{X}) $ if and only if $M \otimes_A A[x_i^{-1}] \simeq M[x_i^{-1}] \simeq 0$ for $1 \le i \le n$ if and only if $M \in \cD_A^{I-\mathrm{tors}}$ (see \Cref{rem:ideal_torsion}).
\end{proof}
Using this, we can given an identification of the local category $\cD_A^{I-\mathrm{loc}}$. We learned that such an approach is possible from \cite[Section 7]{PortaShaulYekutieli2014homology}.
\begin{thm}\label{thm:local}
  Let $\cal{X},\cal{Z}$ and $\cal{U}$ be as above.
  \begin{enumerate}
    \item The functor $j_* \colon \cD_{qc}(\cal{U})  \to \cD_{qc}(\cal{X}) $ is fully-faithful.
    \item There is an equivalence of $\infty$-categories
    \[
\cD_A^{I-\mathrm{loc}} \simeq j_*\cD_{qc}(\cal{U})
    \]
    where the right-hand side denotes the essential image of $j_*$.
  \end{enumerate}
\end{thm}
\begin{proof}\sloppy
  (1) follows by applying the classical flat base-change theorem (see, for example, \cite[Proposition 3.1.3.1]{neeman_fbc}) to the diagram
  \[
\begin{tikzcd}
  \cal{U} \arrow[d,equal] \ar[r,equal] & \cal{U} \arrow[d,"j"] \\
  \cal{U} \arrow[r,"j"'] & \cal{X},
\end{tikzcd}
  \]
  which is a pull-back because $j$ is an open-embedding. Indeed, it implies that the counit $j^*j_* \to \text{id}$ is an equivalence, so that $j_*$ is fully-faithful as claimed.

  Let us write $\cal{E}$ for the essential image of $j_*$. Let ${}^{\bot}\cal{E}$ denote the left orthogonal to $\cal{E}$, i.e., the full subcategory of $\cD_{qc}(\cal{X})$ on those objects $\cal{F}$ for which  $\Hom_{\cD_{qc}(\cal{X})}(\cal{F},\cal{G}) \simeq 0$ for each $\cal{G} \in \cal{E}$. Such a $\cal{G}$ is by definition of the form $j_*\cal{H}$ for $\cal{H} \in \cD_{qc}(\cal{U})$. The vanishing condition is then equivalent to $\Hom_{\cD_{qc}(\cal{U})}(j^*\cal{F},\cal{H}) \simeq 0$ for each $\cal{H} \in \cD_{qc}(\cal{U})$, which is equivalent to $j^*\cal{F} \simeq 0$. Thus $\cal{F} \in \cD^{\cal{Z}}_{qc}(\cal{X}) \simeq \cD_A^{I-\mathrm{tors}}$ by \Cref{lem:torsion_alg_geo}, and so ${}^{\bot}\cal{E} \simeq \cal{D}^{I-\mathrm{tors}}$. It follows from observations about semi-orthogonal decompositions (in particular, \cite[Corollaries 7.1.2.7 and 7.1.2.8]{sag}) that $\cal{E} \simeq \cD_A^{I-\mathrm{loc}}$ as claimed.
\end{proof}
\subsection{Torsion and complete objects for ring spectra}
We now consider the case where $\cC = \Mod_R$ for a commutative ring spectrum $R$ with $\pi_*R$ Noetherian. Suppose we are given an ideal $I = (x_1,\ldots,x_n) \subseteq \pi_*R$. We first construct natural analogs of the Koszul complexes we constructed for graded rings.

To that end, for $x \in \pi_*R$ we let $\mathbb{K}(x)$ be the fiber of the map $\Sigma^{|x|}R \xr{x} R$, and then define the unstable Koszul complex as
\[
\mathbb{K}(I) = \bigotimes_{i=1}^n \mathbb{K}(x_i).
\]
We then define $\Mod_R^{I-\mathrm{tors}}$ to be the category of torsion objects with respect to the compact object $A = \mathbb{K}(I)$, and so we also obtain categories $\Mod_R^{I-\mathrm{loc}}$ and $\Mod_R^{I-\mathrm{comp}}$.

 We also define $\mathbb{K}_{\infty}(x)$ to be the fiber of $R \to R[1/x]$, and then
\[
\mathbb{K}_{\infty}(I) = \bigotimes_{i=1}^n \mathbb{K}_{\infty}(x_i).
\]
The following is implicit in the proof of \cite[Proposition 9.3]{DwyerGreenleesIyengar2006Duality}.
\begin{prop}\label{prop:completions}
  Suppose that $k$ is a field, $R$ is a coconnective commutative augmented $k$-algebra, and that $\pi_*R$ is Noetherian, such that the augmentation induces an isomorphism $\pi_0R \cong k$. Let $I$ denote the augmentation ideal, then there is a symmetric monoidal equivalence of $\infty$-categories
  \[
\Mod_R^{I-\mathrm{comp}} \simeq_{\otimes} L_k\Mod_R,
  \]
  where $L_k\Mod_R$ is the Bousfield localization of $\Mod_R$ at $k$ in the category of $R$-modules.
\end{prop}
\begin{proof}
  By \Cref{lem:bosfield_localization} we have $\Mod_R^{I-\mathrm{comp}} \simeq L_{\mathbb{K}(I)}\Mod_R$, so
  it suffices to show that there is an equivalence of Bousfield classes $\langle k \rangle = \langle \mathbb{K}(I) \rangle$, i.e., that for any $M \in \Mod_R$ we have $k \otimes_R M \simeq 0$ if and only if $\mathbb{K}(I)\otimes_R M \simeq 0$. It is clear that $\pi_*\mathbb{K}(I)$ is finite dimensional over $k$, and hence by \cite[Proposition 3.16]{DwyerGreenleesIyengar2006Duality} $\mathbb{K}(I)$ is in the thick subcategory of $R$-modules generated by $k$ (note that it is here where the conditions on $R$ and $k$ are required). This easily implies that if $k \otimes_R M \simeq 0$, then $\mathbb{K}(I) \otimes_R M \simeq 0$.

  For the converse, we first claim that $k$ is in the localizing subcategory generated by $\mathbb{K}(I)$. Indeed, $k \otimes_R \mathbb{K}_{\infty}(I) \simeq \Gamma_I(k) \simeq k$ by \Cref{rem:torsion_functor}, and so $k \in \Loc_R(\mathbb{K}(I))$. Once again, a simple argument now shows that if $\mathbb{K}(I) \otimes_R M \simeq 0$, then $k \otimes_R M \simeq 0$. This completes the proof.
\end{proof}
\section{Equivariant homotopy theory}
In this section we study the stable equivariant category of a compact Lie group $G$. To that end, we let $\Sp_G$ be the symmetric monoidal $\infty$-category of genuine $G$-spectra for $G$ a compact Lie group, see \cite[Section 5]{MathewNaumannNoel2017Nilpotence}, which is based on the model theoretic foundations of Mandell and May \cite{MellMay2002Equivariant}. This is compactly generated by the set $\{G/H_+ \in \Sp_G\}_{H \le G}$ where $H \le G$ is a closed subgroup (we are omitting the suspension from our notation). Moreover, these objects are dualizable by \cite[Corollary II.6.3]{lms}. The category $\Sp_G$ is closed-monoidal, and we will let $F(-,-)$ denote the internal hom object in $G$-spectra.
\subsection{Change of group functors}\label{sec:cog}
There are a variety of functors in use in equivariant homotopy. Here we recall what we need. Details can be found in, for example, \cite{lms} or Appendix A of \cite{HillHopkinsRavenel2016nonexistence} or \cite[Chapter 3]{Schwede2018Global}.
\begin{enumerate}
  \item Any group homomorphism $f \colon H \to G$ induces a symmetric monoidal functor $f^* \colon \Sp_G \to \Sp_H$. If $f$ is the inclusion of a subgroup, then we denote this as $\Res_H^G \colon \Sp_G \to \Sp_H$. Note that if $H$ is the trivial subgroup, then $\Res_{\{e\}}^G$ is in fact a functor $\Sp_G \to \Fun(BW_GK,\Sp)$, where $W_GK = N_GK/K$ is the Weyl group of $K$ inside $G$.
  \item Restriction has a left adjoint, given by induction. Specifically, $\Ind_H^G \colon \Sp_H \to \Sp_G$ is given by $X \mapsto G_+ \wedge_H X$ for $X \in \Sp_H$. 
  \item If $f \colon G \to G/N$ is a quotient map associated to a normal subgroup $N \trianglelefteq G$, then $f^*$ is the inflation functor $\Sp_{G/N} \to \Sp_G$.
  \item The right adjoint to inflation is the categorical fixed point functor $(-)^N \colon \Sp_G \to \Sp_{G/N}$. If $K \le G$ is an arbitrary subgroup, we let $(-)^K \colon \Sp_G \to \Sp_{W_GK}$ denote the composite $(-)^K \circ \Res_{N_GK}^K$.
  \item For a normal subgroup $N \trianglelefteq G$ we have a geometric fixed points functor $\Phi^N \colon \Sp_G \to \Sp_{G/N}$ (see also \Cref{rem:geo_fixed_points} for a direct construction). If $K \le G$ is an arbitrary subgroup, we write $\Phi^K \colon \Sp_G \to \Sp_{W_GK}$ for the composite $\Phi^{N_GK}\circ \Res_{N_GK}^G$. We also let $\phi^K \colon \Sp_G \to \Fun(BW_GK,\Sp)$ denote the composite $\res_{\{ e \}}^{W_GK}\circ \Phi^K$. It is not hard to check that $\phi^K \simeq \Phi^K \circ \res_K^G$, where we again observe that $\Phi^K \colon \Sp_K \to \Sp$ has a residual action by the Weyl group $W_GK$ (see \cite[Remark 3.3.6]{Schwede2018Global}). By \cite[Proposition 3.3.10]{Schwede2018Global} the functors $\{ \phi^K \}$ as $K$ runs through the closed subgroups of $G$ are jointly conservative. These also have the property that
  \begin{equation}\label{eq:fixed_points}
\phi^H(\Sigma^{\infty}_+X) \simeq \Sigma^{\infty}_+(X^H)
  \end{equation}
  for any $G$-space $X$ and that they are symmetric monoidal, colimit preserving functors.

\end{enumerate}
\subsection{Torsion and complete objects for genuine equivariant \texorpdfstring{$G$}{G}-spectra}\label{sec:torsion_spectra}
We now review the construction of the category of free and cofree (or Borel complete) $G$-spectra in the context of torsion and complete objects as studied in \Cref{sec:torsion_completion}.

We recall the definition of a family of subgroups.
\begin{defn}
  A family of closed subgroups is a non-empty collection $\cal{F}$ of closed subgroups of $G$ closed under conjugation and passage to subgroups.
\end{defn}
Associated to $\cal{F}$ are $G$-spaces $E\cal{F}$ and $\widetilde{E}F$ with the properties that
\begin{equation}\label{eq:family_fixed_points}
(E\cal{F})^H = \begin{cases}
  \emptyset &\text{ if }H \not \in \cal{F} \\
  \ast & \text{ if } H \in \cal{F}.
\end{cases}
\quad \text{ and } \quad
(\widetilde{E}\cal{F})^H = \begin{cases}
  S^0 &\text{ if }H \not \in \cal{F} \\
  \ast & \text{ if } H \in \cal{F}.
\end{cases}
\end{equation}
In fact, the $G$-spaces $E\cal{F}$ and $\widetilde E\cal{F}$ are determined up to homotopy by their behavior on fixed points \cite[Theorem 1.9]{Luck2005Survey}.

Associated to these spaces is a cofiber sequence of pointed $G$-spaces
\begin{equation}\label{eq:isotropy}
E\cal{F}_+ \to S^0 \to \widetilde{E}\cal{F}.
\end{equation}
We will also let $E\cal{F}_+$ and $\widetilde{E}\cal{F}$ denote the suspension spectra of the same pointed $G$-space.
\begin{ex}
  \begin{enumerate}
    \item If $\cal{F}_e = \{ \{ e \} \}$, the family consisting only of the trivial subgroup, then a model for $E\cal{F}_e$ is the universal $G$-space $EG$.
    \item If $\cal{F} = \mathrm{All}$, the family of all closed subgroups of $G$, then a model for $E\cal{F}$ is a point.
  \end{enumerate}
\end{ex}
Given a family $\cal{F}$ we let $A_{\cal{F}} = \{ G/H_+ \mid H \in \cal{F} \}$.
\begin{defn}
  A $G$-spectrum $X$ is $\cal{F}-$\emph{torsion} if it is $\cal{A}_{\cal{F}}$-torsion (i.e., in the localizing subcategory of $\Sp_G$ generated by $A_{\cal{F}}$),\footnote{In this case, this is automatically a localizing ideal by the Mackey decomposition formula.} is $\cal{F}$-local if it is $\cal{A}_{\cal{F}}$-local, and is $\cal{F}$-complete if it is $\cal{A}_{\cal{F}}$-complete.
\end{defn}

The situation can be shown diagrammatically as follows.
\begin{equation}
\xymatrix{& \Sp_{G}^{\cal{F}-\text{loc}} \ar@<0.5ex>[d] \ar@{-->}@/^1.5pc/[ddr] \\
& \Sp_{G} \ar@<0.5ex>[u]^{-[A_{\cal{F}}^{-1}]} \ar@<0.5ex>[ld]^{\Gamma_{A_{\cal{F}}}} \ar@<0.5ex>[rd]^{\Lambda^{A_{\cal{F}}}} \\
\Sp_{G}^{\cal{F}-\mathrm{tors}} \ar@<0.5ex>[ru] \ar[rr]_{\sim} \ar@{-->}@/^1.5pc/[ruu] & & \Sp_{G}^{\cal{F}-\mathrm{comp}}. \ar@<0.5ex>[lu]}
\end{equation}
The following is essentially the content of \cite[Section 4]{Greenlees2001Tate}. For finite $G$, see also
\cite[Propositions 6.5 and 6.6]{MathewNaumannNoel2017Nilpotence}.
\begin{prop}\label{prop:completion_identification}
  The $A_{\cal{F}}$-torsion, localization, and completion functors are given by
  \[
  \begin{split}
\Gamma_{A_{\cal{F}}} &\simeq E\cal{F}_+ \otimes -\\
 -[A_{\cal{F}}^{-1}]& \simeq \widetilde{E}\cal{F} \otimes - \\
\Lambda^{A_{\cal{F}}}&\simeq F(E\cal{F}_+,-).
\end{split}
  \]
\end{prop}
\begin{proof}
  For finite $G$ this is \cite[Theorem 8.6]{BarthelHeardValenzuela2018Local}, however the same proof works for a compact Lie group. Indeed, the key observation is due to Greenlees \cite[Section 4]{Greenlees2001Tate}, who shows that $\Gamma_{A_{\cal{F}}}(S_G) = E\cal{F}_+$. Because $\Gamma_{A_{\cal{F}}}$ is smashing, this determines its behavior on all of $\Sp_G$. The identification of $-[A_{\cal{F}}^{-1}]$ then comes from comparing the cofiber sequences of \Cref{thm:hps}(2) and \eqref{eq:isotropy}, while local duality (\Cref{thm:hps}(4)) gives the identification of $\Lambda^{A_{\cal{F}}}$.
\end{proof}
\begin{defn}\label{defn:free_cofree}
  $X$ is said to be free (respectively, cofree) if it is $A_{\cal{F}}$-torsion (respectively, $A_{\cal{F}}$-complete) for the family $\cal{F} = \{ \{ e \} \}$ consisting only of the trivial subgroup.
\end{defn}
The following is \cite[Proposition 6.19]{MathewNaumannNoel2017Nilpotence} in the case when $G$ is a finite group. The same proof works for compact Lie groups, with the exception that we only need to use closed subgroups because $\{ G/H_+ \in \Sp_G\}_{H \le G}$ is a set of generators for $\Sp_G$, where $H \le G$ is a closed subgroup.

\begin{prop}\label{prop:borel_cofree}
  Suppose $X$ is a $G$-spectrum with underlying spectrum with $G$-action $X_u \in \Fun(BG,\Sp)$. Then the following are equivalent:
  \begin{enumerate}
    \item $X$ is cofree, i.e., the natural map $X \to F(EG_+,X)$ is an equivalence in $\Sp_G$.
    \item For each closed subgroup $H \le G$ the map $X^H \to X_u^{hH}$ is an equivalence of spectra.
  \end{enumerate}
\end{prop}

We now introduce an alternative model of cofree $G$-spectra. For finite $G$, this is \cite[Proposition 6.17]{MathewNaumannNoel2017Nilpotence} or \cite[Theorem II.2.7]{NikolausScholze2018topological}, where for the latter we use \Cref{prop:borel_cofree} to identify Scholze and Nikolaus' Borel-complete $G$-spectra with cofree spectra. The latter proof generalizes to compact Lie groups.
\begin{prop}\label{prop:cofree_global}
  There are equivalences of symmetric monoidal $\infty$-categories \[
  \Sp_G^{\mathrm{cofree}} \simeq_{\otimes} \Fun(BG,\Sp) \quad \text{ and } \quad \Sp_{G,\Q}^{\mathrm{cofree}} \simeq_{\otimes} \Fun(BG,\Mod_{H\Q}).
  \]
\end{prop}
\begin{proof}
We explain the global case; the rationalized case is identical. We first observe that there is a natural functor $\Sp_G \to \Fun(BG,\Sp)$, see \cite[p.~249]{NikolausScholze2018topological}. Alternatively, this is just the observation that the restriction from $\Sp_G \to \Sp$ naturally lands in $\Fun(BG,\Sp)$.

Using \Cref{prop:completion_identification} the same argument\footnote{To be precise, one needs the analog of the equivalence of (i) and (ii) in Theorem 7.12 of \cite{schwede_equivariant} used in \cite{NikolausScholze2018topological}. This follows, for example, from \cite[Proposition V.3.2]{MellMay2002Equivariant}.} as in \cite[Theorem II.2.7]{NikolausScholze2018topological} shows that the functor $\Sp_G \to \Fun(BG,\Sp)$ factors over $\Lambda^G$ (which is the functor denoted $L$ by Nikolaus--Scholze) and that, moreover, the functor $\Sp_G^{\mathrm{cofree}} \to \Fun(BG,\Sp)$ has an inverse equivalence $B_G \colon \Fun(BG,\Sp) \to \Sp_G^{\mathrm{cofree}}$.  Finally, the equivalence is symmetric monoidal, because the induced functor $\Sp_G^{\mathrm{cofree}} \to \Fun(BG,\Sp)$ is symmetric monoidal.
\end{proof}

\subsection{The category of $G$-spectra at $K$}
We now construct a category of $G$-spectra `at $K$', where $K$ is a closed subgroup of $G$. If $K = \{ e \}$ is the trivial subgroup, then this will just be the category of cofree $G$-spectra, while if $K = G$ itself, then this will be equivalent to the ordinary category of non-equivariant spectra.

\begin{defn}
  For a closed subgroup $K \le G$, let $\cF_{\not \ge K}$ denote the family of closed subgroups $H$ of $G$ such that $K$ is not subconjugate to $H$. This defines a localized category $\Sp_G[A_{\cF_{\not \ge K}}^{-1}]$. Additionally, let $\cF_{\le K}$ denote the family of closed subgroups $H$ that are subconjugate to $K$, and $\cF_{<K}$ the family of proper subgroups subconjugate to $K$.
\end{defn}

If we let $(H)$ denote the conjugacy class of a closed subgroup $H \le G$, and write $(H) \le (K)$ when $H$ is subconjugate to $G$, then we can write
\[
\begin{split}
\cal{F}_{\not \ge K} &= \{ H \le G \mid (H) \not \ge (K) \}\\
\cal{F}_{\le K} &= \{ H \le G \mid (H) \le (K) \}\\
\cal{F}_{< K} &= \{ H \le G \mid (H) \lneq (K) \}.
\end{split}
\]
\begin{rem}
  If $K$ is a closed normal subgroup, then $\Sp_G[A_{\cF_{\not \ge K}}^{-1}]$ is known as the category of $G$-spectra concentrated over $K$, see \cite[Chapter II.9]{lms}.
\end{rem}
\begin{lem}\label{lem:fixed_points}
  The following are equivalent for a $G$-spectrum $X$:
  \begin{enumerate}
    \item $X \in \Sp_G[A_{\cF_{\not \ge K}}^{-1}]$.
    \item $\phi^H(X) \simeq 0$ for all $H \in \cF_{\not \ge K}$.
  \end{enumerate}
\end{lem}
\begin{proof}
  See \cite[Lemma 3.20]{quig_shah} for the finite group case, although the argument holds equally well in the case of compact Lie groups. For the benefit of the reader, we spell the details out.

  If (1) holds, then $X \to \widetilde{E}\cal{F}_{\not \ge K} \otimes X$ is an equivalence by \Cref{prop:completion_identification}. Given that $\phi^H$ is symmetric monoidal, \eqref{eq:fixed_points} and the behavior of fixed points of $\widetilde{E}\cal{F}_{\ge K}$ (see \eqref{eq:family_fixed_points}) show that (2) then must hold. Conversely, suppose that (2) holds. To show that (1) holds, it suffices to show that $X \otimes E\cal{F}_{\not \ge K} \simeq 0$. By \cite[Proposition 3.3.10]{Schwede2018Global} we can test this after applying $\phi^H$, as $H$ runs through the closed subgroups of $G$. We then have
  \[
\phi^H(X \otimes E\cal{F}_{\not \ge K}) \cong \phi^H(X) \otimes \phi^H(E\cal{F}_{\not \ge K}) \cong \phi^H(X) \otimes (E\cal{F}_{\not \ge K})^H
  \]
By assumption (2) and \eqref{eq:family_fixed_points} this is always trivial, as required.
\end{proof}
The following is \cite[Corollary II.9.6]{lms} in the global case, and the rational case follows with an identical argument.
\begin{prop}[Lewis--May--Steinberger]\label{prop:ns}
  Let $G$ be a compact Lie group, then for any closed normal subgroup $N \trianglelefteq G$ categorical fixed points induces equivalences of symmetric monoidal $\infty$-categories
  \[
\Sp_G[A_{\cF_{\not \ge N}}^{-1}] \simeq_{\otimes} \Sp_{G/N} \quad \text{ and } \quad \Sp_{G,\Q}[A_{\cF_{\not \ge N}}^{-1}] \simeq_{\otimes} \Sp_{G/N,\Q}.
  \]
\end{prop}
More specifically, the (non-rationalized) equivalence is given as the composite
\[
\Sp_G[A_{\cF_{\not \ge N}}^{-1}] \subseteq \Sp_G \xr{(-)^N} \Sp_{G/N}
\] with inverse given by inflation followed by the localization.
\begin{rem}\label{rem:geo_fixed_points}
  The geometric fixed points functor $\Phi^N \colon \Sp_G \to \Sp_{G/N}$ is defined as the composite
  \[
\Sp_G \xr{-\otimes \widetilde E \cal{F}_{\not \ge N}} \Sp_G[A_{\cF_{\not \ge N}}^{-1}] \subseteq \Sp_G \xr{(-)^N} \Sp_{G/N}.
  \]
  In general, the above composite makes sense for arbitrary $K \le G$, and defines a functor $\widetilde \Phi^K \colon \Sp_G \to \Sp_{W_GK}$. We claim that $\widetilde \Phi^K \simeq \Phi^K$, where the latter is defined in \Cref{sec:cog}. In order to make the dependence on the group clear, we write $\cal{F}_{\not \ge K}^G = \{ H \le G \mid (H) \not \ge (K) \}$ and
$\cal{F}_{\not \ge K}^{N_GK} = \{ H \le N_GK \mid (H) \not \ge (K) \}$.
The two are related by $\cal{F}_{\not \ge K}^{N_GK} = \cal{F}^G_{\not \ge K} \cap \Sub(N_GK)$, where $\Sub(N_GK)$ is the set of closed subgroups of $N_GK$. It is also then not hard to check using fixed points that $\Res_{N_GK}^G( \widetilde E\cal{F}^{G}_{\not \ge K})$ is a model for $\widetilde E\cal{F}^{N_GK}_{\not \ge K}$.

   To see that the two functors are the same, we first claim that $\Res^G_{N_GK}\colon \Sp_G \to \Sp_{N_GK}$ restricts to a functor $\Res^G_{N_GK} \colon \Sp_G[A_{\cal{F}^G_{\not \ge K}}^{-1}] \to \Sp_{N_G K}[A_{\cal{F}^{N_GK}_{\not \ge K}}^{-1}]$ between the localized categories. Let $M \in \Sp_G[A_{\cal{F}^G_{\not \ge K}}^{-1}]$, then by \Cref{lem:fixed_points} we must show that $\phi^H(\Res^G_{N_GK}M) \simeq 0$ for all $H \in \cal{F}^{N_GK}_{\not \ge K}$. By the definition of $\phi^H$, we have
\[
\begin{split}
  \phi^H(\Res^G_{N_GK}M) & \cong \Phi^H \res^{N_GK}_H\res_{N_GK}^GM \\
  & \cong \Phi^H \res_H^GM \\
  & \cong \phi^H M.
\end{split}
\]
Since $H \in \cal{F}_{\not \ge K}^{N_GK}$ we see that $H \in \cal{F}_{\not \ge K}^G$ as well. By \Cref{lem:fixed_points} and the assumption on $M$, we deduce that $ \phi^H(\Res^G_{N_GK}M) \cong \phi^HM \cong 0$, as required.

   It now follows that the diagram
   \[
\begin{tikzcd}[column sep=5em]
\Sp_G \arrow[r, "-\otimes \widetilde{E}\cal{F}^{G}_{\not \ge K}"] \arrow[d, "\Res^G_{N_GK}"'] & {\Sp_G[A_{\cal{F}^G_{\not \ge K}}^{-1}]} \arrow[r, hook] \arrow[d, "\Res^G_{N_GK}"'] & \Sp_G \arrow[r, "(-)^K"] \arrow[d, "\Res^G_{N_GK}"'] & \Sp_{W_GK} \arrow[d, Rightarrow,no head ] \\
\Sp_{N_GK} \arrow[r, "- \otimes \widetilde E \cal{F}^{N_GK}_{\not \ge K}"]                    & {\Sp_{N_GK}[A_{\cal{F}^{N_GK}_{\not \ge K}}^{-1}]} \arrow[r, hook]                   & \Sp_{N_GK} \arrow[r, "(-)^K"']                       & \Sp_{W_GK}
\end{tikzcd}
   \]
   commutes; the first square commutes by the discussion above, the middle square is clear, and the third square commutes by definition of $(-)^K$. This is precisely the claim that $\widetilde \Phi^K \simeq \Phi^K$.
\end{rem}

As noted in \cite[Remark 3.28]{quig_shah} a set of compact generators for $\Sp_{G}[A_{\cal{F}_{\not \ge K}}^{-1}]$ is given by $\{G/H_+ \otimes \widetilde{E}\cal{F}_{\not \ge K} \mid H \not \in \cal{F}_{\not \ge K} \text{ a closed subgroup}\}$ (this also follows from the fact that the localization is smashing and \Cref{prop:completion_identification}).

\begin{defn}
  Let $\Sp_{G,{\langle K \rangle}}$ denote the localizing subcategory of $\Sp_{G}[A_{\cal{F}_{\not \ge K}}^{-1}]$ generated by $\{G/H_+ \otimes \widetilde{E}\cal{F}_{\not \ge K} \mid H \in \cal{F}_{\le K}\}$, and let $\Sp_G^{\langle K \rangle}$ denote the corresponding complete category.
\end{defn}
Of course, we can make similar definitions in the rational case. Diagrammatically the situation is as follows.
\begin{equation}
\xymatrix{& \Sp_{G}^{\langle K \rangle-\text{loc}} \ar@<0.5ex>[d] \ar@{-->}@/^1.5pc/[ddr] \\
& \Sp_{G}[A_{\cal{F}_{\not \ge K}}^{-1}] \ar@<0.5ex>[u] \ar@<0.5ex>[ld] \ar@<0.5ex>[rd] \\
\Sp_{G,\langle K \rangle} \ar@<0.5ex>[ru] \ar[rr]_{\sim} \ar@{-->}@/^1.5pc/[ruu] & & \Sp_{G}^{\langle K \rangle}. \ar@<0.5ex>[lu]}
\end{equation}
\begin{rem}
  By \cite[Lemma 3.25]{quig_shah}, we could also first localize with respect to the family $\cal{F}_{< K}$ instead of $\cal{F}_{\ge K}$. This follows because $\cal{F}_{<K} = \cal{F}_{\ge K} \cap \cal{F}_{\le K}$.
\end{rem}
\begin{lem}\label{lem:gfp_atk}
  A non-trivial $G$-spectrum $X$ is in $\Sp_{G,{\langle K \rangle}}$ if and only if
  \[
\phi^H(X) =0 \text{ if } H \ne K \quad \text{ and } \quad  \phi^H(X) \ne 0  \text { if } H = K,
  \]
  as $H$ runs through the conjugacy classes of subgroups of $G$. In other words, the geometric isotropy of $X$ is exactly $K$.
\end{lem}
\begin{proof}
  We have already seen that $X \in \Sp_G[A_{\cal{F}_{\not \ge K}}^{-1}]$ if and only if $\phi^H(X) \simeq 0$ for all $H \in \cal{F}_{\not \ge K}$. A similar argument shows that $X \in \Sp_{G,\langle K \rangle}$ if and only if $\phi^H(X) \simeq 0$ for the set $\{ H \mid H \in \cal{F}_{\not \ge K}\text{ or } H \not \in \cal{F}_{\le K}\}$. This set contains all the subgroups of $G$ except for $K$. Finally, note that because $X$ is non-trivial, we must have $\phi^K(X) \ne 0$ by \cite[Proposition 3.3.10]{Schwede2018Global}.
  \end{proof}
  \begin{rem}\label{rem:categories}
    The categories $\Sp_G^{\langle K \rangle}$ and $\Sp_{G,\langle K \rangle}$ appear naturally in the work of Ayala--Mazel-Gee--Rozenblyum \cite{1910.14602} and Balchin--Greenlees \cite{1903.02669}. In fact, \Cref{prop:cor_atk} proved below is essentially the identification of the $K$-th stratum of $\Sp_G$, in the sense of Ayala--Mazel-Gee--Rozenblyum, as the category $\Fun(BW_GK,\Sp)$. Such a result is also obtained in \cite[Theorem 5.1.26]{1910.14602}. Using \Cref{lem:gfp_atk} one sees that the rational category $\Sp_{G,\Q,\langle K\rangle}$ also appears in Greenlees' computation of the localizing tensor ideals of $\Sp_{G,\Q}$ \cite{Greenlees2019Balmer}, where it is denoted $G$-spectra$\langle K \rangle$. Greenlees proves that these are precisely the minimal localizing tensor ideals in $\Sp_{G,\Q}$.
  \end{rem}
We let $T \colon \Sp_{G}[A_{\cal{F}_{\not \ge K}}^{-1}] \to \Sp_{W_GK}$ denote the composite
\[
\Sp_G[A_{\cal{F}_{\not \ge K}}^{-1}] \subseteq \Sp_G \xrightarrow{\Res^G_{N_GK}} \Sp_{N_GK} \xr{(-)^K} \Sp_{W_GK}
\]
This is a composite of right adjoints, and so has a left adjoint $F$ given as the composite
\[
\Sp_{W_GK} \xr{\Infl_{W_GK}^{N_GK}} \Sp_{N_GK} \xr{\Ind_{N_GK}^G} \Sp_{G} \xr{ - \otimes \widetilde{E}\cal{F}_{\not \ge K} } \Sp_G[A_{\cal{F}_{\not \ge K}}^{-1}].
\]
Explicitly, the left adjoint takes $L \in \Sp_{W_GK}$ to $(G_+ \otimes_{(N_GK)_+} L)\otimes \widetilde{E}\cal{F}_{\not \ge K}$

\begin{thm}\label{thm:gspectraatk}
  The functor $T \colon \Sp_{G}[A_{\cal{F}_{\not \ge K}}^{-1}] \to \Sp_{W_GK}$ induces equivalences of symmetric monoidal stable $\infty$-categories
  \[
\Sp_{G}^{\langle K \rangle} \simeq_{\otimes} \Sp_{W_GK}^{\mathrm{cofree}} \quad \text{ and } \quad \Sp_{G,\mathbb{Q}}^{\langle K \rangle} \simeq_{\otimes} \Sp_{W_GK,\Q.}^{\mathrm{cofree}}
  \]
\end{thm}
\begin{proof}
We use the compactly generated localization principle \Cref{prop:compact_gen_loc} applied to the adjunction
\[
\begin{tikzcd}
F \colon \Sp_{W_GK} \arrow[r,shift left=1] & \Sp_{G}[A_{\cal{F}_{\not \ge K}}^{-1}] \arrow[l, shift left=1] \colon T
\end{tikzcd}
\]
Here $T(M) = M^K$ and $F(L) = (G_+ \otimes_{(N_GK)_+} L)\otimes \widetilde{E}\cal{F}_{\ge K}$.

  Note that the category $\Sp_{G}^{\langle K \rangle}$ is compactly generated by the object
  \[
\{ G/H_+ \otimes \widetilde{E}\cal{F}_{\not \ge K} \mid H \not \in \widetilde{E}\cal{F}_{\not \ge K},H \in \cal{F}_{\le K}  \} = \{ G/K_+ \otimes \widetilde{E}\cal{F}_{\not \ge K} \}
\]
For simplicity we let $E'$ denote this object. The category $\Sp_{W_GK}^{\mathrm{cofree}}$ is compactly generated by $(W_GK)_+$. Hence, it suffices to show that $T(E') = (W_GK)_+$ and that $FT(E') \to E'$ is an equivalence.\footnote{Note that the two conditions in part (b) of \Cref{prop:compact_gen_loc}(2) are equivalent in this case.} The second in fact follows from the first condition, as then
\[
FT(E') \simeq (G_+ \otimes_{ (N_GK)_+} (N_GK/K)_+) \otimes \widetilde E\cal{F}_{\ge K} \simeq E',
\]
and one checks using the triangle identities that $FT(E') \to E'$ is indeed an equivalence.

 Finally, for the first equivalence, we argue similar to the proof of Theorem 3.22 of \cite{1910.14602}; we have equivalences
  \begin{align*}
  T(E') = T(G/K_+ \otimes \widetilde{E}\cal{F}_{\ge K})  &= (G/K_+ \otimes \widetilde{E}\cal{F}_{\ge K})^K &  \\
  & \simeq \Phi^K(G/K_+) & [\mathrm{\Cref{rem:geo_fixed_points}}] \\
  & \simeq ((G/K)^K)_+ & [\textrm{Equation } \mathrm{\eqref{eq:fixed_points}}]\\
  & \simeq (W_GK)_+ &%
\end{align*}
where the last step uses that $(G/K)^K = W_GK$ as $W_GK$-spaces.

Thus, the assumptions of \Cref{prop:compact_gen_loc}(2) are satisfied and show that
\[
L_{ (W_GK)_+} \Sp_{W_GK} \simeq_{\otimes} L_{G/K_+ \otimes \widetilde{E}\cal{F}_{\not \ge K}}\Sp_{G}[A_{\cF_{\not \ge K}}^{-1}]
\]
By \Cref{lem:bosfield_localization} this is the statement that
  \[
\Sp_{W_GK}^{\mathrm{cofree}}\simeq_{\otimes}  \Sp_{G}^{\langle K \rangle}
  \]
  as claimed.
  \end{proof}

By local duality, or by a similar argument using the cellularization principle (\Cref{prop:cellularization_principle}(2)), we deduce the following.
\begin{cor}\label{cor:gspectraatkfree}
      The functor $T \colon \Sp_{G}[A_{\cal{F}_{\not \ge K}}^{-1}] \to \Sp_{W_GK}$ induces equivalences of stable $\infty$-categories
  \[
\Sp_{G,\langle K \rangle} \simeq \Sp_{W_GK}^{\mathrm{free}} \quad \text{ and } \Sp_{G,\langle K \rangle,\Q} \simeq \Sp_{W_GK,\Q}^{\mathrm{free}}
  \]
\end{cor}

By combining \Cref{thm:gspectraatk} with \Cref{prop:cofree_global} we obtain the following.
\begin{cor}\label{prop:cor_atk}
  Let $G$ be a compact Lie group and $K$ a closed subgroup, then there are equivalences of symmetric monoidal stable $\infty$-categories
  \[
\Sp_{G}^{\langle K \rangle} \simeq_{\otimes} \Fun(BW_GK,\Sp) \quad \text{ and }  \quad \Sp_{G,\Q}^{\langle K \rangle} \simeq_{\otimes} \Fun(BW_GK,\Mod_{H\Q}).
  \]
\end{cor}
It is worthwhile commenting on the two extreme cases: if $K = \{ e \}$, the trivial subgroup, then $BW_GK \simeq BG$, $\Sp_G^{\langle \{ e \} \rangle}$ is the category of cofree $G$-spectra, and the above result is just \Cref{prop:cofree_global}. On the other hand, if $K = G$, then $BW_GK \simeq B\{ e \}$, the one point space, and this is just the obvious equivalence between $\Sp$ and $\Fun(B\{ e \},\Sp)$ that holds more generally for any category.
\section{Unipotence}In this section we review the unipotence criterion of Mathew, Naumann, and Noel \cite{MathewNaumannNoel2017Nilpotence}, and give conditions on $E$ that ensure that $\Loc_{E}(BX)$ is unipotent for a connected finite loop space $X$.
\subsection{A unipotence criterion}
Throughout this section we fix a presentable symmetric monoidal stable $\infty$-category $(\cC,\otimes,\unit)$. We recall that there is an adjunction
\begin{equation}\label{eq:adjunction}
\begin{tikzcd}
\Mod_{R} \arrow[r,shift left=1] & \cC  \arrow[l, shift left=1]
\end{tikzcd}
\end{equation}
where $R = \End_{\cC}(\unit)$, the left adjoint is the symmetric monoidal functor given by $ - \otimes_{R} \unit$ and the right adjoint is given by $\Hom_{\cC}(\unit,-)$.

\begin{defn}[{\cite[Definition 7.7]{MathewNaumannNoel2017Nilpotence}}]
  $\cC$ is unipotent if \eqref{eq:adjunction} is a localization, i.e., if the right adjoint is fully faithful.
\end{defn}

For us, the most important result will be the following unipotence criterion \cite[Proposition 7.15]{MathewNaumannNoel2017Nilpotence}.
\begin{prop}\label{prop:unipotence_criterion}
Let $\cC$ be a presentable symmetric monoidal stable $\infty$-category $(\cC,\otimes,\unit)$. Suppose $\cC$ contains an algebra object $A \in \Alg(\cC)$ with the following properties:
  \begin{enumerate}
    \item $A$ is compact and dualizable in $\mathcal{C}$.
    \item $\mathbb{D}A$ is compact and generates $\mathcal{C}$ as a localizing subcategory.
    \item The $\infty$-category $\Mod_{\mathcal{C}}(A)$ is generated by $A$ itself, and $A$ is compact in $\Mod_{\mathcal{C}}(A)$.
    \item The natural map
    \[
\Hom_{\mathcal{C}}(\unit,A) \otimes_R \Hom_{\cal{C}}(\unit,A) \to \Hom_{\cal{C}}(\unit,A \otimes A)
    \]
    is an equivalence, where $R = \End_{\cal{C}}(\unit)$.
  \end{enumerate}
  Then $\cC$ is unipotent. More specifically, the adjunction \eqref{eq:adjunction} gives rise to a symmetric monoidal equivalence of $\infty$-categories
\[
\cC \simeq_{\otimes} L_{A_R}\Mod_R
\]
  where $A_R =  \Hom_{\cC}(\unit,A) \in \Alg(\Mod(R))$ and the Bousfield localization is taken in the category of $R$-modules.
\end{prop}

\begin{rem}\label{rem:unipotence_cg}
  We now show how to recover the unipotence criterion \Cref{prop:unipotence_criterion} from the compactly generated localization principle \Cref{prop:compact_gen_loc}. In fact, the proof of the unipotence criterion uses \cite[Proposition 7.13]{MathewNaumannNoel2017Nilpotence}, so we assume the existence of a commutative algebra object $A$ satisfying the following:
\begin{enumerate}
  \item $A$ is compact and dualizable in $\cC$.
  \item $\mathbb{D}A$ generates $\cC$ as a localizing subcategory.
  \item $A$ belongs to the thick subcategory generated by the unit.
\end{enumerate}
Assuming these three conditions, we show how to use the compactly generated localization principle to deduce that $\cC \simeq_{\otimes} L_{A_R}\Mod_R$.

We will apply \Cref{prop:compact_gen_loc} to the adjunction
\[
\begin{tikzcd}
F \colon \Mod_{R} \arrow[r,shift left=1] & \cC \colon G \arrow[l, shift left=1]
\end{tikzcd}
\]
where $R = \End_{\cC}(\unit)$, the left adjoint is the symmetric monoidal functor given by $F= - \otimes_{R} \unit$ and the right adjoint is given by $G=\Hom_{\cC}(\unit,-)$.

   We let $E=A_R = G(A) \simeq \Hom_{\cC}(\unit,A)$, then the counit $A_R \to GF(A_R)$ is an equivalence. Indeed, $GF(A_R) \simeq GFG(A) \simeq G(A) \simeq A_R$, see the first paragraph of the proof of \cite[Proposition 7.13]{MathewNaumannNoel2017Nilpotence} (which uses assumption (3)), and the counit $A_R \to GF(A_R)$ is then the identity map by the triangle identities. Moreover, $L_{F(A_R)}\cC \simeq L_A{\cC} \simeq \cC$ by the second paragraph of the proof, which uses assumption (2).

   By \cite[Proposition 2.27]{MathewNaumannNoel2017Nilpotence} $\mathbb{D}A_R$ is a compact generator for $L_{A_R}$ (this uses assumption (1)), and $F(\mathbb{D}A_R) \simeq \mathbb{D}A$ is a compact generator of $L_{F(A_R)}\cC \simeq \cC$ by assumption (2). Thus, applying \Cref{prop:compact_gen_loc} we deduce that there is an equivalence of symmetric monoidal stable $\infty$-categories
   \[
L_{A_R}\Mod_R \simeq_{\otimes} \cC
   \]
   as claimed by the unipotence criterion.
\end{rem}
\subsection{Unipotence for local systems}
We begin be recalling the definition of local systems on a space.
\begin{defn}
  Let $E$ be a commutating ring spectrum, then for $Y$ a connected space, we let $\Loc_{E}(Y) = \Fun(Y,\Mod(E))$ be the $\infty$-category of $E$-valued local systems on $Y$. This is a presentable symmetric monoidal stable $\infty$-category, where the monoidal structure is given by the pointwise tensor product.
\end{defn}
\begin{rem}
  With $E$ and $Y$ as above we also define spectra
  \[
  C^*(Y;E) = F(\Sigma^{\infty}_+Y,E) \quad \text{ and } \quad C_*(Y;E) = \Sigma^{\infty}_+ Y \otimes E.
  \]
  Note that because $E$ is a commutative ring spectrum, so is $C^*(Y;E)$, via the diagonal map.  If $E = H\Q$, we will simply write $C^*(Y;\Q)$ and $C_*(Y;\Q)$. 
\end{rem}
We will usually be interested in the case where $E = H\Q$, but there is no harm in working more generally for now.

Let $e \colon \ast \to Y$ correspond to a choice of base-point for the connected space $Y$. By the adjoint functor theorem, the symmetric monoidal pullback functor $e^* \colon \Loc_E(Y) \to \Loc_E(\ast) \simeq \Mod_E$ has a left and right adjoint, denoted $e_!$ and $e_*$ respectively (these are given by left and right Kan extension along $e$, respectively, see \cite[Section 4.3.3]{Lurie2009Higher}). The following is a special case of \cite[Lemma 4.3.8]{ambidexterity} (recall that we assume $Y$ connected).
\begin{lem}\label{lem:generated}
  The $\infty$-category $\Loc_{E}(Y)$ is generated under colimits by $e_!(E)$.
\end{lem}
\begin{rem}\label{rem:adjoints}
  Suppose more generally that $f \colon X \to Y$ is a map of connected spaces, then there is a symmetric monoidal pull-back functor $f^* \colon \Loc_E(Y) \to \Loc_E(X)$, which, by the adjoint functor theorem, has a left and right adjoint, denoted $f_!$ and $f_*$.
\end{rem}
We now introduce the class of spaces we are most interested in.
\begin{defn}
  A connected finite loop space is a triple $(X,BX,e)$ where $X$ is a connected finite $CW$-complex, $BX$ is a pointed space, and $e \colon X \to \Omega BX$ is an equivalence.
\end{defn}

We will often just refer to the finite loop space as $X$. To apply the unipotence criteria we need to discuss the relevance of the Eilenberg--Moore spectral sequence. We recall the definition from \cite{MathewNaumannNoel2017Nilpotence} here.
\begin{defn}
  Let $Y$ be a space and $E$ a commutative ring spectrum. We say that the $E$-based Eilenberg--Moore spectral sequence (EMSS) is relevant for $Y$ if the square
  \[
\begin{tikzcd}
  C^*(Y;E) \ar[r] \ar[d] & \ar[d]E\\
  E \ar[r] & C^*(\Omega Y;E)
\end{tikzcd}
  \]
  is a pushout of commutative ring spectra, i.e., the induced map $E \otimes_{C^*(Y;E)}E \to C^*(\Omega Y;E)$ is an equivalence.
 \end{defn}

Finally, we need the following, which is a special case of a definition in \cite[Section 8.11]{DwyerGreenleesIyengar2006Duality}.
\begin{defn}
  We say that $C^*(Y;E)$ is a Poincar\'e duality algebra if there exists an $a$ such that $C^*(Y;E) \to \Sigma^a\Hom_E(C^*(Y;E);E)$ is an equivalence. In the case that $R = Hk$ for a field $k$, then $C^*(Y;k)$ satisfies Poincar\'e duality if and if $H^*(Y;E)$ satisfies algebraic Poincar\'e duality.
\end{defn}
Now suppose that $Y$ is a finite $CW$-complex, then we have
\[
\begin{split}
\Hom_E(C^*(Y;E);E) &\simeq \Hom_E(C^*(Y;S) \otimes E;E)\\
& \simeq \Hom(C^*(Y;S);E)\\
& \simeq \Hom(C^*(Y;S);S) \otimes E \\
&\simeq C_*(Y;S) \otimes E\\
&\simeq C_*(Y;E).
\end{split}
\]

With these preliminaries in mind, we now have the following, which is strongly inspired by the closely related result \cite[Theorem 7.29]{MathewNaumannNoel2017Nilpotence}.
\begin{rem}
  We recall that given a commutative ring spectrum $R$, we can form Bousfield localization in the category of $R$-modules, see for example \cite[Chapter VIII]{ElmendorfKrizMellMay1997Rings}. In particular, if $E$ is an $R$-module, then $M \in \Mod_R$ is $E$-acyclic if $E \otimes_R W \simeq \ast$, and a map $f \colon S \to T$ of $R$-modules is an $E$-equivalence if $\text{id} \otimes_R f \colon E \otimes_R M \to E \otimes_R N$ is a weak equivalence.  Then, there is always a localization of $M \in \Mod_R$, i.e., a map $\lambda \colon M \to L_EM$ such that $\lambda$ is an $E$-equivalence, and $M_E$ is $E$-local, i.e., $F_E(W,M_E) = 0$ for any $E$-acyclic $R$-module $W$. 

  For the following, we apply this in the case $R = C^*(BX;E)$ with $E$, considered as an $R$-module via the natural augmentation $C^*(BX;E) \to E$.  
\end{rem}
\begin{thm}\label{prop:mnn_unipotence}
  Let $X$ be a connected finite loop space and $E$ a commutative ring spectrum. Suppose that $C^*(X;E)$ is a Poincar\'e duality algebra, then $\Loc_{E}(BX)$ is unipotent if and only if the $E$-based Eilenberg--Moore spectral sequence for $BX$ is relevant. Moreover, if this holds then there is a symmetric monoidal equivalence of $\infty$-categories
  \[
\Loc_{E}(BX) \simeq_{\otimes} L_{E}\Mod_{C^*(BX,E)}
  \]
  where the Bousfield localization is taken in the category of $C^*(BX;E)$-modules.
\end{thm}
\begin{proof}
  We first show that if the $E$-based EMSS for $X$ is relevant, then $\Loc_E(BX)$ is unipotent.  To do this, we will apply the unipotence criteria of Mathew--Naumann--Noel given in \Cref{prop:unipotence_criterion} to the commutative algebra object $A = C^*(X;E)$ in $\cC = \Loc_{E}(BX)$. Throughout, we let $p \colon BX \to \ast$ and $e \colon \ast \to BX$ denote the canonical maps.

Note that $A = e_*(E)$, and that $e_!(E) \simeq C_*(X;E)$. Moreover, the functor $e_!$ preserves compact objects (as its right adjoint $e^*$ preserves small colimits), and so we deduce that we deduce that $C_*(X;E)$ is compact in $\cC$.

We also have
  \[
R = \End_{\cC}(\unit) \simeq \Hom_{\cC}(p^*(E),p^*(E)) \simeq \Hom_{E}(E,p_*p^*(E)) \simeq p_*p^*(E) \simeq C^*(BX;E)
  \]
  and
  \[
A_R = \Hom_{\cC}(\unit,A) \simeq \Hom_{\cC}(\unit,e_*(E)) \simeq \Hom_{E}(e^*(\unit),E) \simeq E.
  \]
      We now show that the (4) conditions of \Cref{prop:unipotence_criterion} are satisfied, which will imply that $\Loc_E(BX)$ is unipotent.
    \begin{enumerate}
      \item Because $e_!(E) \simeq C_*(X;E)$ is compact in $\cC$, the assumption that $C^*(X;E)$ is a Poincar\'e duality algebra implies that $C^*(X;E)$ is also compact.
      \item $\mathbb{D}A \simeq C_*(X;E) \simeq e_!(E)$, and hence is compact. That $\cC$ is compactly generated by $\mathbb{D}A$ follows from \Cref{lem:generated}.
      \item Consider the adjoint pair $(e^*,e_*)$. The left adjoint is given by forgetting the basepoint, and the right adjoint takes $M \in \Mod_E$ to $C^*(X;M)$. Using this, one sees that the projection formula holds, i.e., that
      \[
e_*(M) \otimes N \to e_*(M \otimes e^*(N))
      \]
      is an equivalence for $N \in \cC$ and $M \in \Mod_{E}$. Because $e_!$ and $e_*$ agree up to a shift, $e_*$ commutes with arbitrary colimits. Finally $e_*$ is conservative. We can now apply \cite[Proposition 5.29]{MathewNaumannNoel2017Nilpotence}, which shows the adjunction $(e^*,e_*)$ gives rise to an equivalence of $\infty$-categories $\Mod_{\cC}(A) \simeq \Mod_{E}$, which implies the result because $e_*(E) \simeq C^*(X;E) = A$.
      \item By assumption the $E$-based EMSS for $BX$ is relevant, and hence by \cite[Proposition 7.28]{MathewNaumannNoel2017Nilpotence} the natural map
    \[
\Hom_{\mathcal{C}}(\unit,A) \otimes_R \Hom_{\cal{C}}(\unit,A) \to \Hom_{\cal{C}}(\unit,A \otimes A)
    \]
    is an equivalence.
    \end{enumerate}
Conversely, assume that $\Loc_E(BX)$ is unipotent. By \cite[Corollary 7.19]{MathewNaumannNoel2017Nilpotence} the natural map
    \[
\Hom_{\mathcal{C}}(\unit,A) \otimes_R \Hom_{\cal{C}}(\unit,A) \to \Hom_{\cal{C}}(\unit,A \otimes A)
    \]
    is an equivalence, because $A$ is compact in $\cC$ by (1) above. It follows from \cite[Proposition 7.28]{MathewNaumannNoel2017Nilpotence} that the $E$-based EMSS for $BX$ is relevant.
\end{proof}

\section{Rational cochains and algebraic models}
We now put the results of the previous sections together and construct an algebraic model for $\Loc_{H\Q}(X)$ for a connected finite loop space $X$.

\begin{prop}\label{prop:unipotence_loops}
Let $X$ be a connected finite loop space, then there is a symmetric monoidal equivalence of $\infty$-categories
  \[
\Loc_{H\Q}(BX) \simeq_{\otimes} L_{H\Q}\Mod_{C^*(BX,\Q)}
  \]
  where the Bousfield localization is taken in the category of $C^*(BX;\Q)$-modules.  \end{prop}
  \begin{proof}
This is a consequence of \Cref{prop:mnn_unipotence} in the case $E = H\Q$. Indeed, $\pi_*C^*(X;\Q) \cong  H^{-\ast}(X;\Q) \cong \Lambda_{\Q}(x_1,\ldots,x_r)$, and in particular satisfies algebraic Poincar\'e duality, and hence $C^*(X;\Q)$ is a Poincar\'e duality algebra. Thus, it suffices to show that the Eilenberg--Moore spectral sequence for $BX$ is relevant, but because $BX$ is simply connected and we work over $\Q$,  \cite{Dwyer1974Strong} applies to show this.
  \end{proof}
Applying \Cref{prop:completions} we deduce the following.
\begin{cor}\label{cor:local_systems_cochains}
  Let $X$ be a connected finite loop space, then there are symmetric monoidal equivalence of $\infty$-categories
  \[
 \Loc_{H\Q}(BX) \simeq_{\otimes} \Mod_{C^*(BX;\Q)}^{I-\mathrm{comp}}.
  \]
\end{cor}

In order to identify the right hand side of this equivalence, we begin by first identifying $\Mod_{C^*(BX;\Q)}$ with dg-modules over the graded ring $H^*(BX)$.  In order to do this, we first need a few words on free $E_{\infty}$-algebras. In particular, we recall that the free $E_{\infty}$-ring on a generator $t$ is defined (as a spectrum) by $\mathbb{S}\{ t \} = \oplus_{n \ge 0}B \Sigma_k$, and is characterized by the property that
\[
\Map_{\CAlg(\Sp)}(\mathbb{S}\{ t \},R) \simeq \Omega^{\infty}R
\]
naturally. In particular, given a ring spectrum with a class $x \in \pi_{0}R$, we obtain a map of commutative algebras $\mathbb{S}\{ t \} \to R$ sending the class $t \in \pi_0\mathbb{S}\{ t \}$ to $x$.

More generally, if $A$ is an $E_{\infty}$-ring spectrum, the free $E_{\infty}$-$A$-algebra on a generator $t$ is defined as
\[
A\{ t \} = \Sym^*(A) = \bigoplus (A^{\otimes n})_{h\Sigma_n}
\]
where the $\Sigma_n$ action is by permutation on the factors. If we wish $t$ to have degree $d$, then we can define $A\{ t \} = \Sym^*(\Sigma^dA)$. Iterating this procedure, we can define $A\{ t_1,\ldots,t_n \}$ as $(A\{ t_1,\ldots,t_{n-1}\})\{t_n\}$.
If the degrees of the $t_i$ are all even, then there is a canonical map
\[
A\{ t_1,\ldots,t_n \} \to HA[t_1,\ldots,t_n]
\]
which is not an equivalence in general. Here $H$ is the generalized Eilenberg--Maclane spectrum functor, which is right inverse to the functor $\pi_* \colon \Sp \to \GrAb$. However, in the case that $A = H\Q$ this canonical map is an equivalence because the higher rational homology of symmetric groups is trivial.

We deduce the following.
\begin{lem}\label{lem:free_rational}
  The free $E_{\infty}$-$\mathbb{Q}$-algebra on $n$ generators $x_1,\ldots,x_n$ concentrated in even degrees is $H\Q[x_1,\ldots,x_n]$.
\end{lem}
This is one part of the input into the following proposition.
\begin{prop}
  There is a symmetric monoidal equivalence of $\infty$-categories
  \[
\theta \colon \Mod_{C^*(BX:\Q)} \simeq_{\otimes} \cD_{H^*(BX)}.
  \]
\end{prop}
\begin{proof}
  Because $H^*(BX;\Q) \cong \Q[x_1,\ldots,x_n]$, the universal property of the free $E_{\infty}$-$\mathbb{Q}$-algebra gives a morphism
\[
\phi \colon H\Q[x_1,\ldots,x_n] \to C^*(BX;\Q)
\]
of $E_{\infty}$-ring spectra, which is clearly an equivalence. As such one gets a symmetric monoidal equivalence of $\infty$-categories
\[
\Mod_{C^*(BX;\Q)} \simeq_{\otimes} \Mod_{H\Q[x_1,\ldots,x_n]}.
\]
The latter is equivalent (as a symmetric monoidal $\infty$-category) to $\cD_{H^*(BX)}$ by \cite[Proposition 2.10]{Shipley2007algebra} or \cite[Section 7.1.2]{ha} and we are done.
\end{proof}
Because $\theta$ is symmetric monoidal it preserves the tensor unit, i.e., $\theta(C^*(BX;\Q)) \simeq H^*(BX)$. It follows (again using that $\theta$ is symmetric monoidal) that $\theta(\mathbb{K}(I)) \simeq K(I)$, and one deduces the following.
\begin{cor}\label{lem:completion_compare}
  The equivalence $\theta$ restricts to a symmetric monoidal equivalence of $\infty$-categories
  \[
\theta \colon \Mod_{C^*(BX:\Q)}^{I-\mathrm{cmpl}} \simeq_{\otimes} \cD_{H^*(BX)}^{I-\mathrm{cmpl}}.
  \]
\end{cor}
We now come to our main theorem.
\begin{thm}\label{thm:main}
    Let $X$ be a connected finite loop space, then there is a symmetric monoidal equivalence of $\infty$-categories
  \[
\Loc_{H\Q}(BX) \simeq_{\otimes} \cal{D}(\Mod_{H^*(BX)}^{I-\mathrm{comp}}).
  \]
\end{thm}
\begin{proof}
  Combine \Cref{cor:local_systems_cochains,lem:completion_compare,thm:completion_monoidal_equiv}.
\end{proof}
Using \Cref{cor:gspectraatkfree} we deduce the following result.
\begin{cor}\label{cor:pol_williamson}
  Let $G$ be a compact Lie group and $K$ a closed normal subgroup such that the Weyl group $W_GK$ is a connected compact Lie group. There is an equivalence of symmetric monoidal $\infty$-categories
  \[
\Sp_{G,\Q}^{\langle K \rangle} \simeq_{\otimes} \cD( \Mod_{H^*(B(W_GK))}^{I-\mathrm{comp}})
  \]
\end{cor}
If $K = \{ e \}$, then $W_GK \simeq BG$ and we recover \cite[Theorem 8.4]{pol_williamson}.

Finally, we point out that local duality also gives a model for $\Sp_{G,\langle K \rangle,\Q}$ as well as the local category $\Sp_{G,\Q}^{\langle K \rangle-\mathrm{loc}}$. Indeed, in the following diagram each of the three outer categories on the left is equivalent to the corresponding category on the right:\footnote{Note that the middle categories are definitely not equivalent, however.}
\[
\xymatrix{& \Sp_{G,\Q}^{\langle K \rangle-\text{loc}} \ar@<0.5ex>[d] \ar@{-->}@/^1.5pc/[ddr] \\
& \Sp_{G,\Q}[A_{\cal{F}_{\not \ge K}}^{-1}] \ar@<0.5ex>[u] \ar@<0.5ex>[ld] \ar@<0.5ex>[rd] \\
\Sp_{G,\langle K \rangle,\Q} \ar@<0.5ex>[ru] \ar[rr]_{\sim} \ar@{-->}@/^1.5pc/[ruu] & & \Sp_{G,\Q}^{\langle K \rangle} \ar@<0.5ex>[lu]} \quad
\xymatrix{& \cD_{{H^*(B(W_GK))}}^{I-\mathrm{loc}} \ar@<0.5ex>[d] \ar@{-->}@/^1.5pc/[ddr] \\
& \cD_{H^*(B(W_GK))} \ar@<0.5ex>[u] \ar@<0.5ex>[ld] \ar@<0.5ex>[rd] \\
\cD_{H^*(B(W_GK))}^{I-\mathrm{tors}} \ar@<0.5ex>[ru] \ar[rr]_{\sim} \ar@{-->}@/^1.5pc/[ruu] & & \cD_{{H^*(B(W_GK))}}^{I-\mathrm{comp}}. \ar@<0.5ex>[lu]}
\]
Using the algebraic models constructed in \Cref{thm:tors,thm:local} we deduce the following.
\begin{cor}\label{cor:free_local}
  Let $G$ be a compact Lie group and $K$ a closed normal subgroup such that the Weyl group $W_GK$ is a connected compact Lie group.
  \begin{enumerate}
     \item There is an equivalence of symmetric monoidal $\infty$-categories
  \[
\Sp_{G,{\langle K \rangle}, \Q}\simeq_{\otimes} \cD( \Mod_{H^*(B(W_GK))}^{I-\mathrm{tors}})
  \]
  \item Let $\cal{X} = \Spec(H^*(B(W_GK)))$, $\cal{Z} = \mathcal{V}(I)$, and $\cal{U} = \cal{X} - \cal{Z}$. Then there is an equivalence of $\infty$-categories
  \[
\Sp_{G,\Q}^{\langle K \rangle-\mathrm{loc}} \simeq j_*\cal{D}_{qu}(\cal{U}),
  \]
  where the right-hand side denotes the essential image of the fully-faithful functor $j_* \colon \cD_{qu}(\cU) \to \cD_{qu}(\cal{X})$.
   \end{enumerate}
\end{cor}
\section{An Adams spectral sequence}
In this final section we construct an Adams spectral sequence in the category $\cC=\Loc_{H\Q}(BX)$ when $X$ is a connected finite loop space. We once again fix a graded commutative Noetherian ring $A$. We will denote the abelian category $\Mod_A^{I-\mathrm{comp}}$ of $L_0^I$-complete dg-$A$-modules by $\cA$. As we will see, this category has enough projectives, and so we can construct an $\Ext$ functor, denoted $\widehat \Ext$, in this category. We also a notion of homotopy groups in $\cC$.
\begin{defn}
  For $M \in \cC$, let $\pi^{\cC}_*(M) = \pi_*\Hom_{\cC}(\unit,M)$.
\end{defn}
We also recall that $H^*(X) \cong \Lambda_{\Q}(x_1,\ldots,x_n)$; we say that the rank of a finite connected loop space is the integer $n$. The spectral sequence then takes the following form.
\begin{thm}\label{thm:adss}
Let $X$ be a finite connected loop space, then for $M,N \in \cC$, there is a natural, conditionally and strongly convergent, spectral sequence of $H^*(BX)$-modules with
  \[
E_2^{s,t} \cong \widehat\Ext^{s,t}_{H^*(BX)}(\pi_*^{\cC}M,\pi_*^{\cC}N) \cong \Ext^{s,t}_{H^*(BX)}(\pi_*^{\cC}M,\pi_*^{\cC}N) \implies \pi_{t-s}\Hom_{\cC}(M,N).
  \]
  Moreover, $E_2^{s,t} = 0$ when $s > \rank(X)$.
\end{thm}
As we will see, working with ring spectra makes the construction of such a spectral sequence very simple; it is just an example of the universal coefficient spectral sequence constructed in \cite[Theorem IV.4.1]{ElmendorfKrizMellMay1997Rings}.

 We first observe that $\cA$ has enough projectives; these are the $I_n$-adic completion of free-modules (also known as pro-free modules), see \cite[Theorem A.9 and Corollary A.12]{HoveyStrickl1999Morava}. By \cite[Proposition A.15]{BarthelFrankl2015Completed} if $M \in \Mod_{H^*(BX)}$ is a flat $H^*(BX)$-module, then $L_0^IM$ is projective in $\cA$, and the left derived functors $\mathbb{L}^iL_0^IM \cong 0$ for $i >0$ by \cite[Theorem 4.1]{MR1172439}. Note this implies that $\cA$ has projective dimension equal to the rank of $X$; indeed, suppose $M \in \cA$ and choose a flat resolution of $M \in \Mod_{H^*(BX)}$
\[
0 \to F_n \to \cdots \to F_2 \to F_1 \to F_0 \to M.
\]
A simple inductive argument on the short exact sequences associated to the resolution shows that
\[
0 \to L_0^I(F_n) \to \cdots L_0^I(F_2) \to L_0^I(F_1) \to L_0^I(F_0) \to L_0^I(M) \cong M
\]
is a projective resolution of $M$ in $\cA$. Thus, $\cal{A}$ has projective dimension $n$. See also \cite[Proposition 1.10]{hovey_ss}.

Because $L_0^I$ is left adjoint to the inclusion functor $\cA \to \Mod_{H^*(BX)}$, we deduce the following, see also \cite[Proposition 5.6]{pol_williamson} or \cite[Theorem 1.11]{hovey_ss}.
\begin{prop}\label{prop:ext}
  Let $\widehat \Ext$ denote the Ext-groups in $\cA$, then for $P,S \in \cA$ we have
  \[
\widehat \Ext_{H^*(BX)}(P,S) \simeq \Ext_{H^*(BX)}(P,S).
  \]
\end{prop}
We also have the following, which is proved identically to \cite[Corollary 3.14]{BarthelFrankl2015Completed}.
\begin{lem}\label{lem:lcompleteness}
  Suppose $A \in \Mod_{C^*(BX;\Q)}$, then $A \in \Mod^{I-\mathrm{comp}}_{C^*(BX;\Q)}$ if and only if $\pi_*A \in \cA$.
\end{lem}
Combining the previous two results we deduce the following.
\begin{cor}
  If $M,N \in \cC$, then
  \[
\widehat\Ext_{\cA}(\pi_*^{\cC}M,\pi^{\cC}_*N) \simeq \Ext_{H^*(BX)}(\pi^{\cC}_*M,\pi^{\cC}_*N).
  \]
\end{cor}
\begin{proof}
  By \Cref{cor:local_systems_cochains} we have $\Hom_{\cC}(\unit,M) \in \Mod_{C^*(BX;\Q)}^{I-\mathrm{comp}}$, and so by \Cref{lem:lcompleteness} we deduce $\pi^{\cC}_*(M) \cong \pi_*(\Hom_{\cC}(\unit,M)) \in \cA$. The result follows from \Cref{prop:ext}.
\end{proof}

We now construct the Adams spectral sequence.
\begin{proof}[Proof of \Cref{thm:adss}]
  We recall that there is an equivalence of categories $\cC \simeq_{\otimes} \Mod_{C^*(BX;\Q)}^{I-\mathrm{comp}}$, given by sending $M \in \cC$ to $\Hom_{\cC}(\unit,M) \in \Mod_{C^*(BX;\Q)}^{I-\mathrm{comp}}$. Under this then, we have
  \[
  \begin{split}
\pi_{t-s}(\Hom_{\cC}(M,N)) &\cong \pi_{t-s}\Hom_{C^*(BX;\Q)^{I-\mathrm{comp}}}(\Hom_{\cC}(\unit,M),\Hom_{\cC}(\unit,N)) \\
& \cong \pi_{t-s}\Hom_{C^*(BX;\Q)}(\Hom_{\cC}(\unit,M),\Hom_{\cC}(\unit,N))
\end{split}
  \]
  where the last step uses that $\Mod_{C^*(BX;\Q)}^{I-\mathrm{comp}}\to \Mod_{C^*(BX;\Q)}$ is fully-faithful.

  The universal spectral sequence \cite[Theorem IV.4.1]{ElmendorfKrizMellMay1997Rings} then takes the form
  \[
E_2^{s,t} \cong \Ext^{s,t}_{H^*(BX)}(\pi_*^{\cC}M,\pi_*^{\cC}N)\implies \pi_{t-s}\Hom_{\cC}(M,N).
  \]
  In general this spectral sequence is only conditionally convergent but in this case it is strongly convergent because $E_2^{s,t} = 0$ for $s>n$ since $H^*(BX)$ has projective dimension $n$. Along with \Cref{prop:ext}, this proves the theorem.
\end{proof}
Translating back into equivariant homotopy, we deduce the following.
\begin{cor}
  Suppose $G$ is a compact Lie group, and $K$ a closed subgroup such that the Weyl group $W_GK$ is connected. For $X,Y \in \Sp_{G,\Q}^{\langle K \rangle}$, there is a natural, conditionally and strongly convergent, spectral sequence of $H^*(B(W_GK))$-modules with
  \[
E_2^{s,t} \cong \Ext^{s,t}_{H^*(B(W_GK))}(\pi_*^{W_GK}(X^K),\pi_*^{W_GK}(Y^K)) \implies [X^K,Y^K]_{t-s}^{W_GK}.
  \]
  Moreover, $E_2^{s,t} = 0$ when $s > \dim(W_GK)$.
\end{cor}
When $K = \{ e \}$ is the trivial group we recover the connected case of \cite[Theorem 10.6]{pol_williamson}.

Using that there is an equivalence $\Sp_{G,\langle K \rangle, \Q} \simeq \Mod_{C^*(BX;\Q)}^{I-\mathrm{tors}}$, a similar argument gives the following.
\begin{prop}
	Suppose $G$ is a compact Lie group, and $K$ a closed subgroup such that the Weyl group $W_GK$ is connected. For $X,Y \in \Sp_{G,\langle K \rangle, \Q}$, there is a natural, conditionally and strongly convergent, spectral sequence of $H^*(B(W_GK))$-modules with
  \[
E_2^{s,t} \cong \Ext^{s,t}_{H^*(B(W_GK))}(\pi_*^{W_GK}(X^K),\pi_*^{W_GK}(Y^K)) \implies [X^K,Y^K]_{t-s}
^{W_GK}.
  \]
  Moreover, $E_2^{s,t} = 0$ when $s > \dim(W_GK)$.
\end{prop}
When $K = \{ e \}$ is the trivial group we recover the spectral sequence of Greenlees and Shipley \cite[Theorem 6.1]{GreenleesShipley2011algebraic}.
\appendix
\section{Model categories and \texorpdfstring{$\infty$}{infty}-categories}\label{sec:appendix}
Throughout we work with $\infty$-categories as developed in \cite{ha}. Since much of the existing work on rational models has used model categories, here we present a very short summary of the relationship between model categories and $\infty$-categories. More details can be found in \cite[Section 5.1]{MathewNaumannNoel2017Nilpotence} or \cite[Section 1.3.4]{ha}, as well as \cite[Appendix A]{NikolausScholze2018topological}
\begin{defn}
  Let $\cC$ be a model category, and let $\cC^c$ denote the full subcategory of $\cC$ spanned by the cofibrant objects. The model category $\cC$ presents an $\infty$-category $\underline \cC$, as the $\infty$-categorical (or Dywer--Kan) localization $\underline \cC \coloneqq \cC^c[\cal{W}^{-1}]$, where $\cal{W}$ is the collection of weak equivalences in $\cC^c$.
\end{defn}
\begin{rem}
  If $\cC$ admits functorial factorizations, then we can equivalently define $\underline \cC$ using fibrant objects of $\cC$, of from the fibrant-cofibrant objects of $\cC$, see \cite[Remark 1.3.4.16]{ha}.
\end{rem}

Suppose that $F \colon \cC \leftrightarrows \cD \colon G$ is a Quillen pair, then by the universal property of localizations one obtains functors
\[
\begin{tikzcd}
\underline{F} \colon \underline \cC \arrow[r, shift left] & \underline{\cD} \colon \underline{G}  \arrow[l, shift left]
\end{tikzcd}
\]
between the underlying $\infty$-categories. 

The following result is \cite[Proposition 1.5.1]{MR3460765}
\begin{prop}[Hinich]\label{lem:hinich_adjoint}
  The pair $(\underline{F},\underline{G})$ form an adjoint pair of $\infty$-categories.
\end{prop}
If $\cC$ is a symmetric monoidal model category, then $\underline{\cC}$ is a symmetric monoidal $\infty$-category \cite[Example 4.1.3.6]{ha}. Moreover, if $F$ is a symmetric monoidal left Quillen functor, then $\underline{F}$ is a symmetric monoidal functor, and because $\underline{G}$ is right adjoint to $\underline{F}$ by \Cref{lem:hinich_adjoint}, $\underline{G}$ is lax symmetric monoidal by \cite[Corollary 7.3.2.7]{ha}.
\bibliography{rational}{}

\providecommand{\bysame}{\leavevmode\hbox to3em{\hrulefill}\thinspace}
\providecommand{\MR}{\relax\ifhmode\unskip\space\fi MR }
\providecommand{\MRhref}[2]{%
  \href{http://www.ams.org/mathscinet-getitem?mr=#1}{#2}
}
\providecommand{\href}[2]{#2}
\begin{thebibliography}{EKMM97}

\bibitem[ABGP04]{AndersenBauerGrodalPedersen2004finite}
Kasper K.~S. Andersen, Tilman Bauer, Jesper Grodal, and Erik~Kjaer Pedersen,
  \emph{A finite loop space not rationally equivalent to a compact {L}ie
  group}, Invent. Math. \textbf{157} (2004), no.~1, 1--10. \MR{2135183}

\bibitem[AMGR19]{1910.14602}
David Ayala, Aaron Mazel-Gee, and Nick Rozenblyum, \emph{Stratified
  noncommutative geometry}, arXiv:1910.14602.

\bibitem[Bar17]{Barnes2017Rational}
David Barnes, \emph{Rational {$O(2)$}-equivariant spectra}, Homology Homotopy
  Appl. \textbf{19} (2017), no.~1, 225--252. \MR{3633720}

\bibitem[BCHV19]{bchv1}
Tobias Barthel, Nat\`alia Castellana, Drew Heard, and Gabriel Valenzuela,
  \emph{Stratification and duality for homotopical groups}, Adv. Math.
  \textbf{354} (2019), 106733, 61. \MR{3989930}

\bibitem[Bek00]{MR1780498}
Tibor Beke, \emph{Sheafifiable homotopy model categories}, Math. Proc.
  Cambridge Philos. Soc. \textbf{129} (2000), no.~3, 447--475. \MR{1780498}

\bibitem[BF15]{BarthelFrankl2015Completed}
Tobias Barthel and Martin Frankland, \emph{Completed power operations for
  {M}orava {$E$}-theory}, Algebr. Geom. Topol. \textbf{15} (2015), no.~4,
  2065--2131. \MR{3402336}

\bibitem[BG20]{1903.02669}
Scott Balchin and J.~P.~C. Greenlees, \emph{Adelic models of
  tensor-triangulated categories}, Advances in Mathematics \textbf{375} (2020).

\bibitem[BHV18a]{BarthelHeardValenzuela2018Local}
Tobias Barthel, Drew Heard, and Gabriel Valenzuela, \emph{Local duality for
  structured ring spectra}, J. Pure Appl. Algebra \textbf{222} (2018), no.~2,
  433--463. \MR{3694463}

\bibitem[BHV18b]{bhv1}
\bysame, \emph{Local duality in algebra and topology}, Adv. Math. \textbf{335}
  (2018), 563--663. \MR{3836674}

\bibitem[BHV20]{BarthelHeardValenzuela2020Derived}
\bysame, \emph{Derived completion for comodules}, Manuscripta Math.
  \textbf{161} (2020), no.~3-4, 409--438. \MR{4060488}

\bibitem[BKNP04]{BauerKitchlooNotbohmPedersen2004Finite}
Tilman Bauer, Nitu Kitchloo, Dietrich Notbohm, and Erik Kj\ae~r Pedersen,
  \emph{Finite loop spaces are manifolds}, Acta Math. \textbf{192} (2004),
  no.~1, 5--31. \MR{2079597}

\bibitem[BMR14]{BarthelMayRiehl2014Six}
Tobias Barthel, J.~P. May, and Emily Riehl, \emph{Six model structures for
  {DG}-modules over {DGA}s: model category theory in homological action}, New
  York J. Math. \textbf{20} (2014), 1077--1159. \MR{3291613}

\bibitem[BS13]{BrodmannSharp2013Local}
M.~P. Brodmann and R.~Y. Sharp, \emph{Local cohomology}, second ed., Cambridge
  Studies in Advanced Mathematics, vol. 136, Cambridge University Press,
  Cambridge, 2013, An algebraic introduction with geometric applications.
  \MR{3014449}

\bibitem[DG02]{DwyerGreenlees2002Complete}
W.~G. Dwyer and J.~P.~C. Greenlees, \emph{Complete modules and torsion
  modules}, Amer. J. Math. \textbf{124} (2002), no.~1, 199--220. \MR{1879003}

\bibitem[DGI06]{DwyerGreenleesIyengar2006Duality}
W.~G. Dwyer, J.~P.~C. Greenlees, and S.~Iyengar, \emph{Duality in algebra and
  topology}, Adv. Math. \textbf{200} (2006), no.~2, 357--402. \MR{2200850}

\bibitem[Dwy74]{Dwyer1974Strong}
W.~G. Dwyer, \emph{Strong convergence of the {E}ilenberg-{M}oore spectral
  sequence}, Topology \textbf{13} (1974), 255--265. \MR{394663}

\bibitem[EKMM97]{ElmendorfKrizMellMay1997Rings}
A.~D. Elmendorf, I.~Kriz, M.~A. Mandell, and J.~P. May, \emph{Rings, modules,
  and algebras in stable homotopy theory}, Mathematical Surveys and Monographs,
  vol.~47, American Mathematical Society, Providence, RI, 1997, With an
  appendix by M. Cole. \MR{1417719}

\bibitem[GM92]{MR1172439}
J.~P.~C. Greenlees and J.~P. May, \emph{Derived functors of {$I$}-adic
  completion and local homology}, J. Algebra \textbf{149} (1992), no.~2,
  438--453. \MR{1172439}

\bibitem[GM95]{GreenleesMay1995Generalized}
\bysame, \emph{Generalized {T}ate cohomology}, Mem. Amer. Math. Soc.
  \textbf{113} (1995), no.~543, viii+178. \MR{1230773}

\bibitem[Gre01]{Greenlees2001Tate}
J.~P.~C. Greenlees, \emph{Tate cohomology in axiomatic stable homotopy theory},
  Cohomological methods in homotopy theory ({B}ellaterra, 1998), Progr. Math.,
  vol. 196, Birkh\"{a}user, Basel, 2001, pp.~149--176. \MR{1851253}

\bibitem[Gre06]{greenlees_conjecture}
\bysame, \emph{Triangulated categories of rational equivariant cohomology
  theories}, Oberwolfach Reports 8/2006 (2006), 480--488.

\bibitem[Gre19]{Greenlees2019Balmer}
\bysame, \emph{The {B}almer spectrum of rational equivariant cohomology
  theories}, J. Pure Appl. Algebra \textbf{223} (2019), no.~7, 2845--2871.
  \MR{3912951}

\bibitem[GS11]{GreenleesShipley2011algebraic}
J.~P.~C. Greenlees and B.~Shipley, \emph{An algebraic model for free rational
  {$G$}-spectra for connected compact {L}ie groups {$G$}}, Math. Z.
  \textbf{269} (2011), no.~1-2, 373--400. \MR{2836075}

\bibitem[GS13]{GreenleesShipley2013cellularization}
\bysame, \emph{The cellularization principle for {Q}uillen adjunctions},
  Homology Homotopy Appl. \textbf{15} (2013), no.~2, 173--184. \MR{3138375}

\bibitem[GS14]{GreenleesShipley2014algebraic}
\bysame, \emph{An algebraic model for free rational {$G$}-spectra}, Bull. Lond.
  Math. Soc. \textbf{46} (2014), no.~1, 133--142. \MR{3161769}

\bibitem[GS18]{GreenleesShipley2018algebraic}
\bysame, \emph{An algebraic model for rational torus-equivariant spectra}, J.
  Topol. \textbf{11} (2018), no.~3, 666--719. \MR{3830880}

\bibitem[HHR16]{HillHopkinsRavenel2016nonexistence}
M.~A. Hill, M.~J. Hopkins, and D.~C. Ravenel, \emph{On the nonexistence of
  elements of {K}ervaire invariant one}, Ann. of Math. (2) \textbf{184} (2016),
  no.~1, 1--262. \MR{3505179}

\bibitem[Hin16]{MR3460765}
Vladimir Hinich, \emph{Dwyer-{K}an localization revisited}, Homology Homotopy
  Appl. \textbf{18} (2016), no.~1, 27--48. \MR{3460765}

\bibitem[HKRS17]{HessKcedziorekRiehlShipley2017necessary}
Kathryn Hess, Magdalena K\c{e}dziorek, Emily Riehl, and Brooke Shipley, \emph{A
  necessary and sufficient condition for induced model structures}, J. Topol.
  \textbf{10} (2017), no.~2, 324--369. \MR{3653314}

\bibitem[HL17]{ambidexterity}
Michael Hopkins and Jacob Lurie, \emph{Ambidexterity in {$K(n)$}-local stable
  homotopy theory}, Available at
  \url{https://www.math.ias.edu/~lurie/papers/Ambidexterity.pdf}.

\bibitem[Hov04]{hovey_ss}
Mark Hovey, \emph{Some spectral sequences in {M}orava {$E$}-theory}, Available
  at \url{https://hopf.math.purdue.edu/Hovey/morava-E-SS.pdf}.

\bibitem[HPS97]{HoveyPalmieriStrickl1997Axiomatic}
Mark Hovey, John~H. Palmieri, and Neil~P. Strickland, \emph{Axiomatic stable
  homotopy theory}, Mem. Amer. Math. Soc. \textbf{128} (1997), no.~610, x+114.
  \MR{1388895}

\bibitem[HS99]{HoveyStrickl1999Morava}
Mark Hovey and Neil~P. Strickland, \emph{Morava {$K$}-theories and
  localisation}, Mem. Amer. Math. Soc. \textbf{139} (1999), no.~666, viii+100.
  \MR{1601906}

\bibitem[K{\c{e}}d17]{Kolhkedziorek2017algebraic}
Magdalena K{\c{e}}dziorek, \emph{An algebraic model for rational {${\rm
  SO}(3)$}-spectra}, Algebr. Geom. Topol. \textbf{17} (2017), no.~5,
  3095--3136. \MR{3704254}

\bibitem[LMSM86]{lms}
L.~G. Lewis, Jr., J.~P. May, M.~Steinberger, and J.~E. McClure,
  \emph{Equivariant stable homotopy theory}, Lecture Notes in Mathematics, vol.
  1213, Springer-Verlag, Berlin, 1986, With contributions by J. E. McClure.
  \MR{866482}

\bibitem[L{\"{u}}c05]{Luck2005Survey}
Wolfgang L{\"{u}}ck, \emph{Survey on classifying spaces for families of
  subgroups}, Infinite groups: geometric, combinatorial and dynamical aspects,
  Progr. Math., vol. 248, Birkh\"{a}user, Basel, 2005, pp.~269--322.
  \MR{2195456}

\bibitem[Lur09]{Lurie2009Higher}
Jacob Lurie, \emph{Higher topos theory}, Annals of Mathematics Studies, vol.
  170, Princeton University Press, Princeton, NJ, 2009. \MR{2522659}

\bibitem[Lur17]{ha}
\bysame, \emph{Higher {A}lgebra}, 2017, Draft available from author's website
  as \url{https://www.math.ias.edu/~lurie/papers/HA.pdf}.

\bibitem[Lur18]{sag}
\bysame, \emph{Spectral {A}lgebraic {G}eometry}, 2018, Draft available from
  author's website as
  \url{https://www.math.ias.edu/~lurie/papers/SAG-rootfile.pdf}.

\bibitem[MM02]{MellMay2002Equivariant}
M.~A. Mandell and J.~P. May, \emph{Equivariant orthogonal spectra and
  {$S$}-modules}, Mem. Amer. Math. Soc. \textbf{159} (2002), no.~755, x+108.
  \MR{1922205}

\bibitem[MNN17]{MathewNaumannNoel2017Nilpotence}
Akhil Mathew, Niko Naumann, and Justin Noel, \emph{Nilpotence and descent in
  equivariant stable homotopy theory}, Adv. Math. \textbf{305} (2017),
  994--1084. \MR{3570153}

\bibitem[Nee20]{neeman_fbc}
Amnon Neeman, \emph{Grothendieck duality made simple}, {$K$}-theory in algebra,
  analysis and topology, Contemp. Math., vol. 749, Amer. Math. Soc.,
  Providence, RI, 2020, pp.~279--325. \MR{4087643}

\bibitem[NS18]{NikolausScholze2018topological}
Thomas Nikolaus and Peter Scholze, \emph{On topological cyclic homology}, Acta
  Math. \textbf{221} (2018), no.~2, 203--409. \MR{3904731}

\bibitem[Pos16]{Positselski2016Dedualizing}
Leonid Positselski, \emph{Dedualizing complexes and {MGM} duality}, J. Pure
  Appl. Algebra \textbf{220} (2016), no.~12, 3866--3909. \MR{3517561}

\bibitem[PSY14]{PortaShaulYekutieli2014homology}
Marco Porta, Liran Shaul, and Amnon Yekutieli, \emph{On the homology of
  completion and torsion}, Algebr. Represent. Theory \textbf{17} (2014), no.~1,
  31--67. \MR{3160712}

\bibitem[PW20]{pol_williamson}
Luca Pol and Jordan Williamson, \emph{The left localization principle,
  completions, and cofree {$G$}-spectra}, Journal of Pure and Applied Algebra
  \textbf{224} (2020), no.~11, 106408.

\bibitem[QS19]{quig_shah}
J.~D. Quigley and Jay Shah, \emph{On the parametrized {T}ate construction and
  two theories of real {$p$}-cyclotomic spectra}, arXiv:1909.03920.

\bibitem[Rez18]{rezk_charles}
Charles Rezk, \emph{Analytic completion}, Available at
  \url{https://faculty.math.illinois.edu/~rezk/analytic-paper.pdf}.

\bibitem[Sch16]{schwede_equivariant}
Stefan Schwede, \emph{Lecture notes on equivariant stable homotopy theory},
  Available at
  \url{http://www.math.uni-bonn.de/people/schwede/equivariant.pdf}.

\bibitem[Sch18]{Schwede2018Global}
\bysame, \emph{Global homotopy theory}, New Mathematical Monographs, vol.~34,
  Cambridge University Press, Cambridge, 2018. \MR{3838307}

\bibitem[Shi07]{Shipley2007algebra}
Brooke Shipley, \emph{{$H\Bbb Z$}-algebra spectra are differential graded
  algebras}, Amer. J. Math. \textbf{129} (2007), no.~2, 351--379. \MR{2306038}

\bibitem[{Sta}20]{stacks-project}
The {Stacks project authors}, \emph{The stacks project},
  \url{https://stacks.math.columbia.edu}, 2020.

\end{thebibliography}
\bibliographystyle{amsalpha}
\end{document}